\documentclass[12pt,a4paper]{amsart}

\usepackage[margin=0.7in]{geometry}
\usepackage[T1]{fontenc}
\usepackage[utf8]{inputenc}
\usepackage{lmodern}
\usepackage{amssymb}
\usepackage{upref}
\usepackage{mathrsfs}
\usepackage{nicematrix}
\usepackage{tikz}
\usepackage{hyperref}
\hypersetup{%
    colorlinks = true,
    citecolor  = blue,
    filecolor  = magenta,
    linkcolor  = blue,
    urlcolor   = blue,
    pdfauthor  = {Biagio Cassano \& Giovanni Taglialatela},
    pdftitle   = {The Cauchy problem for the biwave operator},
    pdfpagemode = UseOutlines|UseThumb,
    pdfstartview = FitH,
    bookmarksopen = true,
    bookmarksnumbered = true,
    bookmarksopenlevel = {1},
    linktocpage = true,
    pdfsubject  = {The Cauchy problem for the biwave operator},
    pdfcreator  = {-},
    pdfproducer = {-},
}
\usepackage{cleveref}
\usepackage{verbatim}
\usepackage[normalem]{ulem}
\usepackage{graphicx}

\makeatletter

\def\paragraph{%
	\@startsection{paragraph}{4}{\z@}%
		{1ex\@plus.1ex \@minus.2ex}%
		{-1em}%
		{\normalfont\normalsize\scshape}}

\def\subparagraph{%
	\@startsection{paragraph}{4}{\z@}%
		{1ex\@plus.1ex \@minus.2ex}%
		{-1em}%
		{\normalfont\normalsize\itshape}}

\makeatother

\theoremstyle{plain}
\newtheorem{Theorem}{Theorem}
\newtheorem*{Theorem*}{Theorem}
\newtheorem{Proposition}{Proposition}[section]

\newtheorem{Lemma}[Proposition]{Lemma}
\theoremstyle{definition}

\newtheorem{Definition}{Definition}
\theoremstyle{remark}
\newtheorem{Remark}{Remark}
\newtheorem{Example}{Example}
\newtheorem*{Notations}{Notations}

\everymath{\displaystyle}
\allowdisplaybreaks
\numberwithin{equation}{section}

\newcommand{\labelp}[1]{\tag{\ref{#1}$'$}\label{#1p}}
\newcommand{\labelpp}[1]{\tag{\ref{#1}$''$}\label{#1pp}}

\newcommand{\defeq}{\overset{\textnormal{\,{\tiny def}\,}}{=}}

\newcommand{\C}{\mathbb{C}}
\newcommand{\N}{\mathbb{N}}
\renewcommand{\S}{\mathbb{S}}
\newcommand{\R}{\mathbb{R}}

\newcommand{\e}{\varepsilon}

\let\Re\undefinied
\DeclareMathOperator{\Re}{Re}
\let\Im\undefinied
\DeclareMathOperator{\Im}{Im}
\newcommand{\wh}[1]{\widehat{#1}}
\newcommand{\abs}[1]{\left|#1\right|}
\def\<#1\>{\left\langle#1\right\rangle}

\newcommand{\Dfrac}[2]{{%
  \ooalign{%
    $\genfrac{}{}{1.5pt}0{#1}{#2}$\cr%
    $\color{white}\genfrac{}{}{.6pt}0{\phantom{#1}}{\phantom{#2}}$}%
}}

\newcommand{\zeroT}{[0,T]}
\newcommand{\KK}{\mathcal{K}}

\newcommand{\LL}{\mathsf{L}}
\newcommand{\PP}{\mathsf{P}}
\newcommand{\RR}{\mathsf{R}}
\newcommand{\PPP}[2]{\PP_{#1}^{#2}}

\newcommand{\LG}[1]{\mathrm{L}_{#1}}

\newcommand{\aae}{a_\varepsilon}
\newcommand{\bbe}{b_\varepsilon}
\newcommand{\deltae}{\delta_\varepsilon}

\renewcommand{\AA}{\mathscr{A}}
\newcommand{\CC}{\mathscr{C}}

\newcommand{\QQe}{\mathscr{Q}_\varepsilon}

\newcommand{\EEe}{\mathscr{E}_\varepsilon}
	
\newcommand{\Atwo}{\mathbf{A}}
\newcommand{\Itwo}{\mathbf{1}}
\newcommand{\Ztwo}{\mathbf{0}}
\newcommand{\Qtwo}{\mathbf{Q}}

\newcommand{\Rtwoe}{\mathbf{R}_\e}

\newcommand{\I}{{\normalfont\textsc{i}}}
\newcommand{\II}{{\normalfont\textsc{ii}}}
\newcommand{\III}{{\normalfont\textsc{iii}}}

\begin{document}

\title{On the Gevrey well-posedness for weakly hyperbolic biwave type equations}
\author[B.~Cassano \textit{\&} G.~Taglialatela]%
       {Biagio~Cassano \textit{\&} Giovanni~Taglialatela}
\address{B.~Cassano,
	Department of Mathematics and Physics,
	Universit\`{a} degli  Studi della Campania,
	viale Lincoln 5,
	81100 Caserta,
	Italy}
\email{biagio.cassano@unicampania.it}

\address{G.~Taglialatela,
	Dipartimento di Economia e Finanza,
	Universit\`{a} di~Bari Aldo Moro,
	Largo Abbazia S.~Scolastica,
	70124 Bari,
	Italy}
\email{giovanni.taglialatela@uniba.it}


\begin{abstract}
We study the Cauchy problem for~biwave type equations
with coefficients depending only on time.
Assuming $\mathcal{C}^\kappa$ regularity of~the principal symbol
and suitable Levi conditions on the lower order terms,
we~establish well-posedness in~Gevrey spaces.
\end{abstract}

\maketitle

\section{Introduction}

We~consider the Cauchy problem for fourth order linear
operators with coefficients depending only on the time variable
\begin{align}
&  \PP(t;\partial_t,\partial_x) u(t,x)
          +  \RR(t;\partial_t,\partial_x) u(t,x)
   =  f(t,x) \, ,  &
   &  (t,x) \in \zeroT \times \R^n \, , \label{E-eq}
\intertext{%
with initial conditions}
&  \partial_t^j u(t_0,x) = u_j(x)
\qquad
j=0,\dotsc,3 \, ,  &
   &  x \in \R^n \, , \label{E-IC}
\end{align}
on the initial surface $t = 0$.
In our analysis, $\PP(t;\partial_t,\partial_x)$
is a biwave type operator with time dependent coefficients:
\begin{equation} \label{E-biwave}
\PP(t;\partial_t,\partial_x)
  \equiv  \partial_t^4
          -  2\sum_{|\nu|=2}
             b_\nu(t) \, \partial_t^2 \, \partial_x^\nu
          +  \sum_{|\nu|=4}
             a_\nu(t) \, \partial_x^\nu \, ,
\end{equation}
and $\RR(t;\partial_t,\partial_x)$ is an operator of~order $\le3$:
\begin{equation} \label{E-RR}
\RR(t;\partial_t,\partial_x)
  =  \sum_{j=0}^3 \RR_j(t;\partial_t,\partial_x) \, ,
\qquad
\RR_j(t;\partial_t,\partial_x)
  \equiv  \sum_{k+|\nu|=j} c_{k,\nu}(t) \, \partial_t^k \, \partial_x^\nu \, .
\end{equation}

In particular, we consider in \eqref{E-eq} equations of the type
\begin{equation} \label{E-basic}
\partial_t^4 u(t,x) - 2 \, b(t) \, \partial_t^2 \, \Delta u(t,x)
           + a(t) \, \Delta^2 u(t,x) = f(t,x) \, ,
\end{equation}
where $a$ and $b$ are $\mathscr{C}^\kappa$ functions; if $a\equiv b\equiv 1$,
equation~\eqref{E-basic} reduces to the so-called \emph{biwave} equation
\[
\Box^2 u(t,x) = f(t,x) \, ,
\]
where $\Box = \partial_t^2-\Delta$ is the classical wave operator. 

\smallskip

In this paper we are interested
in~finding suitable conditions on the coefficients
in~\eqref{E-eq} to guarantee well-posedness of~the Cauchy Problem~\eqref{E-eq}--\eqref{E-IC}
in~the Gevrey spaces.
We recall the definition of~the latter ones:

\begin{Definition}
Let $\Omega$ be an open set in~$\R^n$, $s \ge 1$ and $f\in\mathscr{C}^\infty(\Omega)$.
We say that $f$ belongs to the Gevrey space~$\gamma^s(\Omega)$
if for any compact set~$\KK \subset\joinrel\subset \Omega$ there exist
positive constants $C_\KK$ and $R_\KK$ such that
\[
\bigl|\partial_x^\alpha f(x)\bigr|
  \le  C_\KK \, R_\KK^{-|\alpha|} \, |\alpha|!^s \, .
\]
In the following we will note, for short, $\gamma^s=\gamma^s(\R^n)$.
\end{Definition}
Note that $\gamma^1(\Omega)$ is the space of~the real analytic functions in~$\Omega$,
$\gamma^1 = \gamma^1(\R^n)$ is the space of~the entire functions in~$\R^n$.
The index $s$ in $\gamma^s(\Omega)$ acts as a continuous mathematical gauge, bridging the gap between analyticity and $ \mathscr{C}^\infty$-smoothness;
note, however that $\bigcup_{s\ge1} \gamma^s \ne \mathcal{C}^\infty$.

We give a definition of~well-posedness suitable
for the Problem~\eqref{E-eq}--\eqref{E-IC}.

\begin{Definition}
Problem~\eqref{E-eq}--\eqref{E-IC} is said to be~\emph{uniformly well-posed in~$\gamma^s$}
if, for any $t_0\in\zeroT$ and for any $u_0(x),\dotsc,u_3(x)$ in~$\gamma^s(\R^n)$,
it admits a unique solution $u(t,x)$ in~$\mathscr{C}^3\bigl(\zeroT ; \gamma^s(\R^n)\bigr)$.
\end{Definition}

We recall that,
in order the Cauchy problem to be well-posed in~$\mathscr{C}^\infty$
or~in Gevrey classes,
the operator should be~\emph{weakly hyperbolic}
\cite{Lax-1957,Mizohata-1961,IvriiPetkov-1974,Nishitani-1978}:
the characteristic roots,
that is the solutions in~$\tau$ of~the equation
\[
\PP(t;\tau,\xi)=0 \, ,
\label{eq:characteristic.roots}
\]
are real for any $t \in \zeroT$ and $\xi \in \R^n$.

In particular, as can be seen
by an explicit computation, a biwave operator as in~\eqref{E-biwave}
is weakly hyperbolic if the following conditions hold true:
\begin{equation} \label{E-WH}
b(t;\xi) \ge 0 \, ,
\quad
a(t;\xi) \ge 0 \, ,
\quad
\delta(t;\xi) \ge 0 \, ,
\quad
\text{for any $t \in \zeroT$ and $\xi \in \S^{n-1}$} \, \tag{H},
\end{equation}
where
\begin{align*}
b(t;\xi)
&  \defeq  \sum_{|\nu|=2} b_\nu(t) \, \bigl(\xi/|\xi|\bigr)^\nu \, ,  \\
a(t;\xi)
&  \defeq  \sum_{|\nu|=4} a_\nu(t) \, \bigl(\xi/|\xi|\bigr)^\nu \, ,  \\
\delta(t;\xi)
&  \defeq  b^2(t;\xi) - a(t;\xi) \, ,
\intertext{%
and}
\S^{n-1}
&  \defeq  \Bigl\{ \ \xi\in\R^n \ \Bigm| \ |\xi| = 1 \ \Bigr\} \, .
\end{align*}

We are interested in the relation between the regularity of~the coefficients
and the Gevrey spaces in which the Cauchy problem is well-posed.
This issue has been considered in~\cite{CDGS}
for second order wave-type equations of~the form
\begin{equation} \label{E-CP2}
u_{tt} - \sum_{j,k=1}^n a_{jk}(t) u_{x_jx_k} = 0 \, ,
\qquad
\text{for $(t,x) \in \zeroT\times\R^n$} \, .
\end{equation}
Assuming that the equation is \emph{strictly} hyperbolic,
namely that the characteristic roots are real and simple,
which is equivalent to
\begin{equation} \label{E-sh2}
\sum_{j,k=1}^n a_{jk}(t) \xi_j\xi_k
  \ge  \theta |\xi|^2 \, ,
\qquad
\text{for $(t;\xi) \in \zeroT\times\R^n$} \, ,
\end{equation}
where $\theta>0$,
in~\cite{CDGS} it is proved that
if~the coefficients $\{a_{jk}\}_{j,k=1,\dotsc,n}$ belong
to~$\mathscr{C}^{0,\alpha}\bigl(\zeroT\bigr)$
then the Cauchy problem associated to~\eqref{E-CP2}
is~well-posed in~$\gamma^s$ with
\begin{equation} \label{E-CDGS}
1<s< \frac{1}{1-\alpha} \, .
\end{equation}
They~also showed, by constructing suitable counterexamples,
that the upper bound in~\eqref{E-CDGS} is sharp.
This result has been extended in~\cite{Nishitani-1983} and~\cite{Nishitani-2006}
to~general strictly hyperbolic second order equations and
in~\cite{Jannelli-1985} and~\cite{Cicognani-1998} to~general strictly hyperbolic systems
of any order.
Observe, in particular, that this includes the study of the well-posedness for~\eqref{E-eq}
in the strictly hyperbolic setting.

The situation (of course) worsens if~\eqref{E-CP2} is only \emph{weakly} hyperbolic,
that is if Hypothesis~\eqref{E-sh2} is replaced by
\begin{equation} \label{E-wh2}
\sum_{j,k=1}^n a_{jk}(t) \xi_j\xi_k \ge 0 \, ,
\qquad
\text{for $(t;\xi) \in \zeroT\times\R^n$} \, .
\end{equation}

In fact, in~\cite{CJS} it is proved that if~the coefficients
$\{a_{jk}\}_{j,k=1,\dotsc,n}$ belong to~$\mathscr{C}^{k,\alpha}\bigl(\zeroT\bigr)$
with $k\in\N$ and $0\le \alpha \le 1$,
then the Cauchy problem associated to~\eqref{E-CP2}
is well-posed in~the Gevrey class $\gamma^s$ with
\begin{equation} \label{E-CJS}
1<s<1+\frac{k+\alpha}{2} \, ,
\end{equation}
whereas it is $\mathscr{C}^\infty$ well-posed
if~the coefficients $\{a_{jk}\}_{j,k=1,\dotsc,n}$ are analytic on $\zeroT$.
Here~too, the upper bound in~\eqref{E-CJS} is sharp.

\smallskip

The above result is specific of~the wave-type equation~\eqref{E-CP2}.
Indeed, concerning higher order equations,
we~recall that Ohya and Tarama~\cite{OhyaTarama-1986,OhyaTarama-1987,OhyaTarama-2006}
proved that if~a~weakly hyperbolic operator
has~characteristic roots of~multiplicity at~most~$r$
and
the coefficients are~$\mathscr{C}^{k,\alpha}$ functions in~$t$
($k=0,1$ and $\alpha\in[0,1]$)
and~$\gamma^s$ in~$x$,
then the associated Cauchy problem is well-posed in~$\gamma^s$,
provided that
\begin{equation} \label{E-OT}
1
  \le  s
   <   \min\Bigl(1+\frac{\kappa+\alpha}{r} ,
                   \frac{r}{r-1}\Bigr) \, .
\end{equation}

The Ohya-Tarama result has been extended
in~\cite{Kinoshita-1998b} to block diagonalizable systems,
in \cite{DAnconaKinoshitaSpagnolo-2004}
to $2\times2$ and~$3\times3$ systems with time dependent coefficients,
in~\cite{Yuzawa-2005} to $N\times N$ systems with time dependent coefficients
and finally in~\cite{KajitaniYuzawa-2006a} to general $N\times N$ systems
(see also~\cite{KajitaniYuzawa-2006b} for non-linear $N\times N$ systems).

\smallskip

Note that the second bound in~\eqref{E-OT},
which is not present in~\eqref{E-CJS},
could not be omitted.
Indeed, it is well-known that for general equations
the regularity of~the coefficients is not enough
to~get the well-posedness in~$\mathscr{C}^\infty$
or in~$\gamma^s$ for high values of~$s$.
In fact, even if the coefficients are~constant,
well-posedness in~$\mathscr{C}^\infty$ or in high Gevrey classes
fails to hold without additional conditions.
More precisely, the following condition
on the roots~$\lambda=\lambda(\xi)$ of~the~\emph{full} symbols
\[
\LL(\lambda,\xi)
  \equiv  \PP(\lambda,\xi) + \RR(\lambda,\xi)
\]
is~\emph{necessary and sufficient}
for the Cauchy Problem to be well-posed in~$\mathscr{C}^\infty$
(see~\cite{Garding-1951}~\cite[\textsection 12.5]{Hormander}~\cite{Svensson-1969}):
\begin{align}
&
\textsl{there exists~$C>0$ such that
    $\bigl|\Im\lambda(\xi)\bigr| \le C$,
    for any $\xi\in\R^n$} \, . \tag{G} \label{E-Garding}
\intertext{\endgraf
\noindent
Analogously,
the following condition is~\emph{necessary and sufficient}
for the Cauchy Problem to be well-posed in~$\gamma^s$
(see~\cite{Larsson-1967}):}
&
\textsl{there exists~$C>0$ such that
    $\bigl|\Im\lambda(\xi)\bigr| \le C \, \bigl(1+|\xi|^{\frac{1}{s}}\bigr)$,
    for any $\xi\in\R^n$} \, . \tag{L} \label{E-Larsson}
\end{align}
{}From~\eqref{E-Garding} and~\eqref{E-Larsson}
we see that the Cauchy problem associated to the equation
\begin{equation} \label{E-es1}
\partial_{tt} u - \partial_x u = 0
\end{equation}
is ill-posed in~$\mathscr{C}^\infty$
and in~$\gamma^s$, for $s>2$,
since the solution to the equation
\[
\lambda^2 - \xi = 0
\]
does not verify~\eqref{E-Larsson} for $s>2$
and, \emph{a fortiori},~\eqref{E-Garding}.
Similarly,
we see that the Cauchy problem associated to the equation
\begin{equation} \label{E-es1m}
\partial_t^m u - \partial_x^{m-1} u = 0
\end{equation}
is ill-posed in~$\mathscr{C}^\infty$
and in~$\gamma^s$ for $s>\frac{m}{m-1}$.

Equations~\eqref{E-es1} and~\eqref{E-es1m} above show that
if the characteristic roots have multiplicities,
the~presence of~the lower order terms may affect
the $\gamma^s$-well-posedness for high values of~$s$.
Note that,
if the operator has constant coefficients and is homogeneous,
that is $\RR(\tau,\xi)\equiv0$,
the hyperbolicity is necessary and sufficient
for the well-posedness in~$\mathscr{C}^\infty$ and in~Gevrey classes.
Indeed, as $\LL(\tau,\xi)\equiv\PP(\tau,\xi)$,
the $\lambda=\lambda(\xi)$ in~\eqref{E-Garding} and~\eqref{E-Larsson}
are $1$-homogeneous w.r.t.~$\xi$,
hence~\eqref{E-Garding} and~\eqref{E-Larsson}
reduce\sout{d} to the hyperbolicity of~the principal symbol.
In the variable coefficients case,
the hyperbolicity of~the principal symbol is no longer sufficient
for the well-posedness
even if the coefficients are very regular.
Indeed the Cauchy problem for
\begin{equation} \label{E-es2}
\partial_{tt} v
  +  2 \, t \, \partial_{tx} v
  +  t^2 \, \partial_{xx} v =  0
\end{equation}
is ill-posed in~$\mathscr{C}^\infty$ and in~$\gamma^s$ for $s>2$,
since Equation~\eqref{E-es2} is equivalent to~\eqref{E-es1},
in the sense that $u(t,x)$ solves~\eqref{E-es2} if, and only if,
$v(t,x) = u(t,x+t^2/2)$ solves~\eqref{E-es1}.

\smallskip
\goodbreak

{}From the above considerations,
we~deduce that,
in~order to obtain well-posedness results
(in~$\mathscr{C}^\infty$ or in~$\gamma^s$, for high values of~$s$),
we must assume four type of~conditions:
\begin{itemize}
\item  hyperbolicity of~the principal symbol,
\item  regularity of~the coefficients of~the principal symbol,
\item  algebric conditions on the principal symbol,
\item  algebric conditions on the lower order terms (usually called \emph{Levi conditions}).
\end{itemize}

Second order equations as well as $2\times2$ systems have been studied
in~\cite{DAbbiccoTaglialatela-2009} and~\cite{DAnconaKinoshitaSpagnolo-2008},
whereas
third order equations have been considered
first in~\cite{Kinoshita-1998a} then in~\cite{JannelliTaglialatela-2014}.
For fourth and higher order equations the situations becomes more complicated,
and we have only few results.

In~\cite{ColombiniKinoshita-2002},
Colombini and Kinoshita considered a homogeneous biwave type operator as in~\eqref{E-biwave}
and they gave a refinement of~Ohya-Tarama's Theorem:
instead of~$b \in \mathscr{C}^{0,\alpha}$,
they assumed that $b^2 \in \mathscr{C}^{0,\alpha}$,
which is a slightly weaker condition (cf.~\mbox{{\cite[pag.~40]{ColombiniKinoshita-2002}})}.

In~\cite{KinoshitaSpagnolo-2006},
Kinoshita and Spagnolo considered homogeneous equations of~order $m$
with time depending coefficients.
Besides the hyperbolicity and the $\mathscr{C}^{k,\alpha}$ regularity of~the coefficients
they assumed the so-called Colombini-Orr\`{u} condition (see~\cite{ColombiniOrru-1999})
on the characteristic roots~$\tau_j$:
there exists $C>0$ such that
\begin{equation} \label{E-CO}
\tau_j^2(t;\xi) + \tau_k^2(t;\xi)
     \le  C \bigl(\tau_j(t;\xi) - \tau_k(t;\xi)\bigr)^2 \, ,
\quad
\forall t\in[0,T] \, , \ \forall \xi\in\R^n \, .
\end{equation}
and they proved that the Cauchy problem is well-posed in~$\gamma^s$
provided that
\begin{equation} \label{E-JT2-bound s}
1\le s < 1+\frac{k+\alpha}{2\,(m-1)} \, .
\end{equation}
We remark that,
refining the estimates in~\cite{KinoshitaSpagnolo-2006}
the upper bound in~\eqref{E-JT2-bound s} can be improved
by replacing the order~$m$ of~the operator
with the maximal multiplicity of~the characteristic roots.
However,
the upper bound for the admissible~$s$ in~\eqref{E-JT2-bound s}
is smaller than the upper bound in~\eqref{E-OT} if $m\ge3$.

\smallskip

Other results concerning uniformly diagonalizable systems
have been obtained in~\cite{Tarama-1994,Colombini-Nishitani-1999,Kajitani-1988,ColombiniMetivier-2018},
but these results are not applicable to scalar equations.
Indeed,
when transforming the scalar equation
\[
\PP u \equiv \partial_t^m u + a_1 \, \partial_t^{m-1} \partial_x u + \dotsb + a_m \, \partial_x^m u
   +  \text{l.o.t.} = 0
\]
into a system,
via the usual transformation
$U = \bigl( \partial_x^{m-1} u , \partial_t \partial_x^{m-2} u , \dotsc , \partial_t^{m-1} u \bigr)^{\mathsf{T}}$,
we get the system
\[
\partial_t U - A \partial_x U +  \text{l.o.t.} = 0 \, ,
\]
where the matrix $A$ has the Sylvester form:
\[
\setlength{\arraycolsep}{3pt}
A = \begin{pNiceMatrix}
    0 & 1 & 0 & \Cdots & 0  \\
      & \Ddots & \Ddots & \Ddots & \Vdots  \\
    \Vdots &   &        &   & 0 \\
    0 & \Cdots &  & 0  & 1  \\
    a_m &  & \Cdots &  & a_1
	\end{pNiceMatrix} \, .
\]
The eigenvalues of~$A$ equal the characteristic roots of~$\PP$.
Since each eigenvalue of~a Sylvester matrix has a one-dimensional eigenspace,
as is easily verified,
the matrix fails to be diagonalizable whenever multiple roots are present.

\smallskip

One of~the difficulties in considering higher order equations
is that even the condition of~hyperbolicity
is much more involved than that for second or third order equations.
Indeed, recall that for a generic monic polynomial
\[
p(\tau) = \prod_{j=1}^m (\tau - \tau_j)
\]
the~\emph{discriminant} is defined by
\[
\Delta_m \defeq \prod_{1\le j<\ell\le m} (\tau_j-\tau_\ell)^2 \, ,
\]
and it is nonnegative if all the $\tau_j$s are real.
For polynomial of~order 2 or 3
the nonnegativity of~the discriminant is also sufficient
in order to have only real roots,
but this is no longer true for polynomials of~order $\ge 4$.
Consider, for instance, the polynomial
\[
p(\tau) = \tau^4 + 5\,\tau^2 + 4
\]
which has only complex roots:
\[
\tau_1 = i \, , \qquad
\tau_2 = -i \, , \qquad
\tau_3 = 2\,i \, , \qquad
\tau_4 = -2\,i \, ,
\]
but its discriminant is positive:
\[
\Delta_4 \defeq \prod_{1 \le j < \ell \le 4} (\tau_j-\tau_\ell)^2 = 5184 > 0 \, .
\]
On the other side,
even for the simple biquadratic polynomial
\[
p(\tau) = \tau^4 - 2\,b\,\tau^2 + a \, ,
\]
the hyperbolicity condition cannot be expressed by less than $3$ inequalities:
\[
a\ge0 \, , \qquad
b\ge0 \, , \qquad
b^2-a \ge 0 \, .
\]

For these reasons,
we restrict our study to biwave type equations,
avoiding the general fourth-order case due to its complexity.
This specific choice is well-motivated by the interest of the scientific community
\cite{Arosio-2001,ANP-1992,ANPP-1992,APP-1992a,APP-1992b,KorzyukCheb2007,KorzyukRudzko2025,Yaqian2024}.
In particular, the biwave equation has applications in elasticity theory, 
see~\cite{Hetnarski-2004, Muskhelishvili-2010, BARCELO201533,CASSANO2021528}.
Indeed, consider the~\emph{displacement equation of~motion} or~\emph{Navier's equation}
\begin{equation} \label{eq:Navier}
\rho \partial_t^2 u = \mu \Delta u + (\lambda + 2 \mu) \nabla (\operatorname{div} u) + b \, ,
\end{equation}
being $u = u (x,t) \in \R^3$ the displacement,
$b= b(x,t) \in \R^3$ the force per unit volume by external agents on~$x$ at time $t$,
$\rho > 0$ the density of~the material and $\lambda,\mu \in \R$, $\mu >0$, $\lambda+2\mu>0$
the Lam\'{e} coefficients describing the elasticity properties of~the medium.
If  $g=g(x,t) \in \R^3$ is solution to the biwave equation
\begin{equation}\label{eq:biwave-CKS}
\left(\frac{\rho}{\lambda + 2\mu} \partial_{t}^2 - \Delta   \right)
\left(\frac{\rho}{\mu} \partial_{t}^2  - \Delta \right) g
   =  - \frac{b}{\mu},
\end{equation}
then
\[
u
  :=  \left(\Delta - \frac{\rho}{\lambda + 2\mu} \partial_t^2\right) g
      -  \frac{\lambda + \mu}{\lambda + 2 \mu} \nabla \operatorname{div} g
\]
is the \emph{Cauchy-Kovalevski-Somigliana solution} to~\eqref{eq:Navier}
(see~\cite[Section 7.2 (U3)]{Hetnarski-2004} for more details).
Additionally, the biwave equation (and its perturbed or non-linear variants) frequently appears in condensed matter physics as a time-dependent simplified Ginzburg-Landau-type model for $d$-wave superconductors (typically in the absence of an applied magnetic field), see \cite{feng2010finite} and references therein.

\smallskip

The paper is organized as follows:
Section 2 outlines the hypotheses and states the main results;
Section 3 provides the proofs
and
Section 4 is dedicated to establishing the properties
of the quasi-symmetrizer needed in the proofs.

\section{Statements of~the results}

In this section we present and discuss our results. To state them, we need to introduce some notations.

\begin{Notations} \hfill
\begin{itemize}
\item  For any function $f(t;\xi)$
        we denote $\partial_t f(t;\xi)$ by $f'(t;\xi)$;
        we also omit the $\xi$ variable to simplify the notations.

\item  For $f(t;\xi)$ and~$g(t;\xi)$ positive functions
       we~will write $f\lesssim g$
       (or, equivalently $g\gtrsim f$)
       to~mean that there exists a positive constant $C$
       such that
       \[
       f(t;\xi)
         \le  C \, g(t;\xi) \, ,
       \qquad
       \text{for any $(t;\xi)\in\zeroT\times\R^n$} \, .
       \]

\item  $\mathscr{C}^{k,\alpha}$
    is the usual vector space of~the $k$-times differentiable functions
    whose $k$-th derivative is $\alpha$-H\"{o}lder continuous.
    As customary in~this setting,
    \textit{e.g.}~\cite{DAbbiccoTaglialatela-2009} and~\cite{JannelliTaglialatela-2014},
    for $\kappa\in\R$ a positive number,
    we~denote by $\mathscr{C}^\kappa$
    the vector space of~the functions $\mathscr{C}^{k,\alpha}$
    where
    \[
    k \defeq \max\{ \, j \in\N \, | \, j<\kappa \, \}
    \]
    and $\alpha = \kappa-k$.

    Note that if $\kappa\in\N\setminus\{0\}$,
    then $\mathscr{C}^\kappa = \mathscr{C}^{\kappa-1,1}$.
\end{itemize}
\end{Notations}

Now we discuss the different hypotheses
(in~addition to~the hyperbolicity~\eqref{E-WH})
that we will need in~the statement of~our results.

\subsection*{Regularity hypotheses}

According to~\cite{CJS},
it is natural to assume that the coefficients of~the principal symbol are
in~$\mathscr{C}^\kappa$ to study the Cauchy problem in Gevrey spaces.

As already mentioned, in~\cite{ColombiniKinoshita-2002}
the well-posedness of~the biwave equation is treated
considering $\mathscr{C}^\kappa$ coefficients, with $0<\kappa\le1$.
Moreover, instead of~$b \in \mathscr{C}^\kappa$,
they assumed that $b^2 \in \mathscr{C}^\kappa$,
which is a slightly weaker condition (cf.~\mbox{\cite[pag.~40]{ColombiniKinoshita-2002}}).
Accordingly,
we will require regularity conditions on~$a$ and~$\delta$.

As customary in this setting,
(see e.g.~\cite{DAbbiccoTaglialatela-2009} and~\cite{JannelliTaglialatela-2014})
we can replace the~$\mathscr{C}^\kappa$ regularity
by a~weaker condition.
Indeed, one of~the main tools in~the proof of~the Gevrey well-posedness
is Lemma~1 in~\cite{CJS},
stating that if $f\in\mathscr{C}^\kappa\bigl(\zeroT\bigr)$
is nonnegative,
then $f^{\frac{1}{\kappa}}$ in~absolutely continuous in~$\zeroT$,
and there exists a constant $C$,
depending only on~$\kappa$ and~$T$,
such that
\[
\bigl|f'(t)\bigr|
  \le  \Lambda(t) \, \bigl[f(t)\bigr]^{1 - \frac{1}{\kappa}} \, ,
\qquad
\text{with $\|\Lambda\|_{L^1(\zeroT)}
              \le  C \, \|f\|_{\mathscr{C}^\kappa(\zeroT)}^{\frac{1}{\kappa}}$} \, .
\]

Thus, in~the following,
we shall consider a (generalized) regularity hypothesis of~the form
\begin{equation}
\bigl|f'(t;\xi)\bigr|
  \lesssim  \Lambda(t;\xi) \, \bigl|f(t;\xi)\bigr|^{1-\alpha} \, ,
\end{equation}
with the appropriate $\alpha >0$ and where, here and in the following,
$\Lambda(t;\xi)$ denotes a measurable function such that
\begin{equation} \label{E-Lambda}
\begin{cases}
\text{$\Lambda(t;\xi)$ is 0-homogeneous with respect to $\xi$} \, ,  \\*[2pt]
\Lambda(t;\xi) \ge 1 \quad \text{for any $t\in\zeroT$ and $\xi\in\R^n$} \, , \\*[2pt]
\sup_{\xi\in\S^{n-1}} \int_0^T \Lambda(t;\xi) \, dt \le C < +\infty \, .
\end{cases}
\end{equation}

\subsection*{Conditions on the principal symbol}

In~\cite{JannelliTaglialatela-2011}, a sufficient condition is given 
for the well-posedness in~$\mathscr{C}^\infty$
for homogeneous equations with time depending analytic coefficients;
this condition is also necessary in~space dimension $n=1$ (cf.~\cite{ColombiniOrru-1999}).
A~weaker form of~this condition has been used in~\cite{JannelliTaglialatela-2014}
to prove well-posedness in~Gevrey spaces
for third order equations.

\smallskip

In the case of~the biwave operator,
the condition found in~\cite{JannelliTaglialatela-2011}
boils down to:
\begin{equation} \label{E-check}
\bigl|b'(t;\xi)\bigr|
  \lesssim  \sqrt{\delta(t;\xi) \, }
                     +  \Bigl| \partial_t \sqrt{\delta(t;\xi) \, } \Bigr| \, .
\end{equation}
Condition~\eqref{E-check} means that if, for some fixed~$\xi$,
the function $t \mapsto \delta(t;\xi)$ is not identically zero
and vanishes for~$t=\overline{t}$ at order~$2k$,
then~$b'(t;\xi)$ must vanish for~$t=\overline{t}$ at order~$k-1$.
If we assume that
$b'(t;\xi)$ vanishes for~$t=\overline{t}$ at order less then~$k-1$
the $\mathcal{C}^\infty$-well-posedness cannot be obtained,
and we must aim to Gevrey well-posedness for certain exponents~$s$.
Following~\cite{JannelliTaglialatela-2014},
we consider a condition of~the form:
\begin{equation} \label{E-checkw}
\bigl|b'(t;\xi)\bigr|
  \lesssim   \Lambda(t;\xi) \, \delta^\alpha(t;\xi)  \, .
\end{equation}
for some function $\Lambda(t;\xi)$ verifying~\eqref{E-Lambda}
and some $\alpha \in \left]0,\frac{1}{2}\right[$.
Condition~\eqref{E-checkw} is clearly weaker than~\eqref{E-check}
which can be seen, in some sense, as the limit for $\alpha\nearrow\frac{1}{2}$ of~\eqref{E-checkw}.

\goodbreak

\subsection*{Levi conditions}

Returning to the G\aa rding condition~\eqref{E-Garding},
we show in the Appendix~\ref{A-1} that for a biwave operator as in~\eqref{E-biwave}
this condition can be explicitly stated as follows:
\begin{alignat}{2}
\bigl[c_{2,1}(\xi)\bigr]^2
&  \lesssim  b(\xi)  &  \qquad
&  \bigl[b(\xi) \, c_{2,1}(\xi) + c_{0,3}(\xi)\bigr]^2
  \lesssim  b(\xi) \, \delta(\xi) \label{E-Levicc1}  \\
\bigl[c_{0,3}(\xi)\bigr]^2
&  \lesssim  a(\xi) \, b(\xi)  &
&  \bigl[b(\xi) \, c_{3,0}(\xi) + c_{1,2}(\xi)\bigr]^2
  \lesssim  \delta(\xi) \label{E-Levicc2}  \\
\bigl|c_{0,2}(\xi)\bigr|
&  \lesssim  b(\xi)  &
&  \bigl[c_{1,1}(\xi)\bigr]^2
  \lesssim  b(\xi)  \qquad
\bigl[c_{0,1}(\xi)\bigr]^2
  \lesssim  b(\xi) \, , \label{E-Levicc3}
\end{alignat}
where
\begin{equation} \label{eq:defn.cjl}
c_{j,\ell}(\xi)
  \defeq  \sum_{|\nu|=\ell} c_{j,\nu} \, \bigl(\xi/|\xi|\bigr)^\nu \, .
\end{equation}

The G\aa rding condition~\eqref{E-Garding} remains necessary and sufficient
for $\mathscr{C}^\infty$ well-posedness even if the lower-order terms have variable coefficients~\cite{Dunn-1975,Wakabayashi-1980}.
Additionally, condition~\eqref{E-Garding} is sufficient
when the principal symbol depends on a single spatial variable
and satisfies the analogue of~condition~\eqref{E-CO}
for operators with space-dependent coefficients~\cite{SpagnoloTaglialatela-2022}.

In this paper, we relax conditions~\eqref{E-Levicc1}--\eqref{E-Levicc3}
by raising the right-hand side to a power $\alpha < 1$,
thereby establishing sufficient conditions for well-posedness in Gevrey classes.

\bigskip
\goodbreak

We can now finally state our main result.

\begin{Theorem} \label{T-Gevrey-4}
Consider the Cauchy problem for the biwave equation~\eqref{E-eq},
with initial conditions~\eqref{E-IC},
where $\PP(t;\tau,\xi)$ and~$\RR(t;\tau,\xi)$ are defined
in~\eqref{E-biwave} and~\eqref{E-RR} respectively.

Let $\kappa\ge1$,
and assume that
there exists a function $\Lambda(t;\xi)$ verifying~\eqref{E-Lambda}
such that the following conditions hold true.

Regularity hypotheses:
\begin{align}
a_\nu \, , \, b_\nu
&  \in  AC\bigl(\zeroT\bigr) \, , \tag{R$_1$} \label{E-HGR}  \\
\bigl|a'(t;\xi)\bigr|
&  \lesssim  \Lambda(t;\xi) \, a^{1-\frac{1}{\kappa}}(t;\xi) \, ,
   \tag{R$_2$} \label{E-HGa}  \\
\bigl|\delta'(t;\xi)\bigr|
&  \lesssim  \Lambda(t;\xi) \, \delta^{1-\frac{1}{\kappa}}(t;\xi) \, .
   \tag{R$_3$} \label{E-HGdelta}
\intertext{\endgraf
Hypothesis on the principal part:}
\bigl|b'(t;\xi)\bigr|
&  \lesssim  \Lambda(t;\xi) \,
             \bigl[\delta(t;\xi)\bigr]^{\frac{1}{2}-\frac{1}{\kappa}}  &
&  \text{if $\kappa \ge 2$} \, .
   \tag{$\psi$} \label{E-HGB}
\intertext{\endgraf
Levi conditions:}
c_{k,\nu} \in L^1\bigl(\zeroT\bigr)
&  \quad  k+|\nu| \le 3
   \tag{$\LG{0}$} \label{E-HGL reg}  \\
\bigl|b(t;\xi) \, c_{2,1}(t;\xi) + c_{0,3}(t;\xi)\bigr|
&  \le  \Lambda(t;\xi) \, \bigl[b(t;\xi) \, \delta(t;\xi)\bigr]^{\frac{1}{2}-\frac{2}{3\kappa}}  &
&  \text{if $\kappa \ge \frac{4}{3}$} \, ,
   \tag{$\LG{1}$} \label{E-HGL c21c03}  \\
\bigl|c_{0,3}(t;\xi)\bigr|
&  \le  \Lambda(t;\xi) \, \bigl[a(t;\xi) \, b(t;\xi)\bigr]^{\frac{1}{2}-\frac{2}{3\kappa}}  &
&  \text{if $\kappa \ge \frac{4}{3}$} \, ,
   \tag{$\LG{2}$} \label{E-HGL c03}  \\
\bigl|b(t;\xi) \, c_{3,0}(t;\xi) + c_{1,2}(t;\xi)\bigr|
&  \le  \Lambda(t;\xi) \, \bigl[\delta(t;\xi)\bigr]^{\frac{1}{2}-\frac{1}{k}}  &
&  \text{if $\kappa \ge 2$} \, ,
   \tag{$\LG{3}$} \label{E-HGL c03c12}  \\
\bigl|c_{2,1}(t;\xi)\bigr|
&  \le  \Lambda(t;\xi) \, \bigl[b(t;\xi)\bigr]^{\frac{1}{2}-\frac{2}{\kappa}}  &
&  \text{if $\kappa \ge 4$} \, ,
   \tag{$\LG{4}$} \label{E-HGL c21}  \\
\bigl|c_{0,2}(t;\xi)\bigr|
&  \le  \Lambda(t;\xi) \, \bigl[b(t;\xi)\bigr]^{1-\frac{4}{\kappa}}  &
&  \text{if $\kappa \ge 4$} \, ,
   \tag{$\LG{5}$} \label{E-HGL c02}  \\
\bigl|c_{1,1}(t;\xi)\bigr|
&  \le  \Lambda(t;\xi) \, \bigl[b(t;\xi)\bigr]^{\frac{1}{2}-\frac{4}{\kappa}}  &
&  \text{if $\kappa \ge 8$} \, ,
   \tag{$\LG{6}$} \label{E-HGL c11}  \\
\bigl|c_{0,1}(t;\xi)\bigr|
&  \le  \Lambda(t;\xi) \, \bigl[b(t;\xi)\bigr]^{\frac{1}{2}-\frac{6}{\kappa}}  &
&  \text{if $\kappa \ge 12$} \, ,
   \tag{$\LG{7}$} \label{E-HGL c01}
\end{align}
where the~$c_{j,\ell}(t;\xi)$ are defined by
\begin{equation} \label{E-cjl}
c_{j,\ell}(t;\xi)
  \defeq  \sum_{|\nu|=\ell} c_{j,\nu}(t) \, \bigl(\xi/|\xi|\bigr)^\nu \, .
\end{equation}

Then the Cauchy problem~\eqref{E-eq}--\,\eqref{E-IC} is well-posed in~$\gamma^s$,
for $1<s<1+\frac{\kappa}{4}$.
\end{Theorem}

\begin{Remark}
Note that if $b^2$ and $a$ are $\mathcal{C}^\kappa$ functions,
then~\eqref{E-HGa} and~\eqref{E-HGdelta} are verified.
We consider here the case $\kappa\ge1$:
our result is consistent with \cite{ColombiniKinoshita-2002}
where they consider $\kappa\le1$.
\end{Remark}

We remark that the Levi conditions appear gradually as $\kappa$ increase
(cf.~Remark~4 in~\cite{Kinoshita-1998a}).
Phenomena of~this kind are usual in~this setting
and are reasonable, as the following example shows.

\begin{Example}
Consider the operator with principal symbol $\PP(t;\tau,\xi)=\tau^4$.

This is an operator with constant coefficients principal part
and a unique characteristic root $\tau=0$ with multiplicity~$4$.
According to~\eqref{E-Larsson},
the largest Gevrey class of~well-posedness can be obtained
thanks to the construction of~a Newton polygon:
if $c_{j,k}\not\equiv0$, draw the point $(k,4-j)$ in~the $(x,y)$ plane,
and consider the convex hull of~these points together with the origin
(which corresponds to $c_{4,0} \equiv 1$).
The maximal Gevrey index of~well-posedness
is given by the slope of~the steepest line through the origin
below all points of~the Newton polygon.

If, for instance,
\begin{equation} \label{E-PR}
\PP(t;\tau,\xi) + \RR(t;\tau,\xi)
  =  \tau^4 - c_{1,2}(t) \, \tau \xi^2 - c_{0,1}(t) \xi \, ,
\end{equation}
the Newton polygon is given in Fig.~\ref{F-1}.

\begin{figure}[h]
\[
\includegraphics[width=0.8\linewidth]{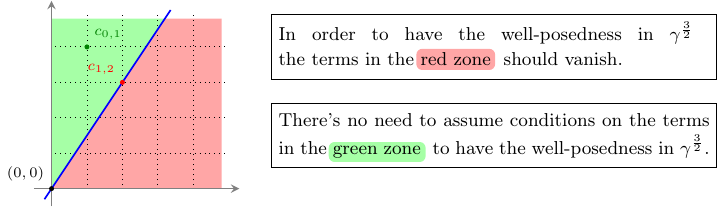}
\]
\caption{Newton polygon for the polynomial in~\eqref{E-PR}}  \label{F-1}
\end{figure}

If, for instance, $c_{0,3}\equiv0$ and $c_{1,2}\not\equiv0$,
then the well-posedness is limited to $\gamma^s$ with $1<s<\frac{3}{2}$,
and the conditions on the other terms are superfluous.
\end{Example}

We remark that the multiplicity of~the characteristic roots is at most double
if, and only if,
\begin{equation} \label{E-b gg 0}
b(t;\xi) \ge \beta_0 \, ,
\quad
\text{for any $(t;\xi) \in \zeroT \times \R^n$} \, ,
\end{equation}
for some $\beta_0>0$.

\smallskip
\goodbreak

In this case Theorem~\ref{T-Gevrey-4} can be improved as follows

\begin{Theorem} \label{T-Gevrey-2}
Consider the Cauchy problem for the biwave equation~\eqref{E-eq},
with initial conditions~\eqref{E-IC},
where $\PP$ and~$\RR$ are defined as in~\eqref{E-biwave} and~\eqref{E-RR}
rispectively.

Let $\kappa\ge1$,
and assume the same Hypothesis~\eqref{E-HGR},~\eqref{E-HGa},~\eqref{E-HGdelta}
and~\eqref{E-HGB} of~Theorem~\ref{T-Gevrey-4}.

Assume moreover that~\eqref{E-b gg 0} is satisfied
and that there exists a function $\Lambda(t;\xi)$ verifying~\eqref{E-Lambda}
such that the following Levi conditions hold true:
\begin{align}
\bigl|b(t;\xi) \, c_{2,1}(t;\xi) + c_{0,3}(t;\xi)\bigr|
&  \le  \Lambda(t;\xi) \, \bigl[\delta(t;\xi)\bigr]^{\frac{1}{2}-\frac{1}{k}}  &
&  \text{if $\kappa \ge 2$} \, ,
   \tag{L$'_1$} \label{E-HGL2 c21c03}  \\
\bigl|c_{0,3}(t;\xi)\bigr|
&  \le  \Lambda(t;\xi) \, \bigl[a(t;\xi)\bigr]^{\frac{1}{2}-\frac{1}{k}}  &
&  \text{if $\kappa \ge 2$} \, ,
   \tag{L$'_2$} \label{E-HGL2 c03}  \\
\bigl|b(t;\xi) \, c_{3,0}(t;\xi) + c_{1,2}(t;\xi)\bigr|
&  \le  \Lambda(t;\xi) \, \bigl[\delta(t;\xi)\bigr]^{\frac{1}{2}-\frac{1}{k}}  &
&  \text{if $\kappa \ge 2$} \, .
   \tag{L$'_3$} \label{E-HGL2 c03c12}
\end{align}
where $\Lambda=\Lambda(t;\xi)$ is a nonnegative measurable function verifying~\eqref{E-Lambda}.

Then the Cauchy problem~\eqref{E-eq}-\eqref{E-IC} is well-posed in~$\gamma^s$,
for $1<s<1+\frac{\kappa}{2}$.
\end{Theorem}

\begin{Remark}
In Theorem~\ref{T-Gevrey-2} the Levi conditions involves only the coefficients of~the terms of~order~3.

This is consistent with the well-known fact (maybe observed by Peyser \cite{Peyser-1969} for the first time)
that the well-posedness of~a differential operator of~order $m$ with characteristic roots of~multiplicity $r$
is stable under perturbation by operators of~order $\le m-r$.
\end{Remark}

\section{Proofs of \Cref{T-Gevrey-4} and \Cref{T-Gevrey-2}}

Our proof follows the~strategy of~\cite{CDGS} and \cite{CJS}
and relies on the method of~approximated energies
and the quasi-symmetrizer~\cite{Jannelli-1989}, \cite{DAnconaSpagnolo-1998}
(see also~\cite{Jannelli-2008}).

\smallskip

First of~all we remark that,
thanks to the Duhamel principle, we can assume that $f(t,x)$ in~\eqref{E-eq} vanishes identically.
Secondly,
we can assume that the initial data~\eqref{E-IC} have compact support;
the result for non compact initial data can be obtained
using the finite speed of~propagation property
(cf.~\cite{CDGS} \cite{Rauch-2005} \cite{ColombiniRauch-2011}).

\subsection*{Proof of~Theorem~\ref{T-Gevrey-4} for homogeneous equations}

We initially consider the case in which the equation is homogeneous,
that is $\RR(t;\tau,\xi) \equiv 0$.

\smallskip

Let $v(t;\xi) \defeq \mathscr{F}_{x\to\xi}\bigl(u(t,x)\bigr)$
be the Fourier transform with respect to the space variables $x$ of~$u(t,x)$,
and $\wh{u}_j(\xi) \defeq \mathscr{F}_{x\to\xi}\bigl(u_j(x)\bigr)$
be the Fourier transform of~the initial data~\eqref{E-IC}.
Problem~\eqref{E-eq}--\eqref{E-IC} is transformed into
the ordinary differential Cauchy problem in~$\zeroT$
\begin{equation} \label{PdCTrasformato}
\begin{cases}
v'''' - 2 \, b(t;\xi) \abs{\xi}^2 v'' + a(t;\xi) \abs{\xi}^4 v  = 0  \\
\partial_t^j v(0,\xi)=\wh{u}_j(\xi)
\qquad
j=0,\dotsc,3 \, ,
\end{cases}
\end{equation}
parameterized by~$\xi\in\R^n$.

We recall  the Paley-Wiener Theorem in~Gevrey spaces
(cf.~\cite[pag.~517]{CDGS} and \cite[Thm.~1.6.1]{Rodino}).

\begin{Theorem*}
Let $\varphi \in \mathscr{C}^\infty_0$, with support contained in the ball $\bigl\{ \, x\in\R^n \, \bigm| \, |x|\le R \, \bigr\}$.
Then $\varphi\in\gamma^s$,
if and only if there exist $C,\eta>0$
such that
\begin{equation} \label{E-PVG}
\bigl|\widehat{\varphi}(\xi+i\eta)\bigr|
  \le  C \, \exp\bigl( -\eta \, \abs{\xi}^{1/s} + R\,|\eta| \bigr) \, ,
\quad
\text{for all $\xi\in\R^n$} \, .
\end{equation}
\end{Theorem*}

Thus, to prove that if $u_j \in \gamma^s$, for $j=0,1,2,3$,
then $u \in \mathscr{C}^3\bigl(\zeroT ; \gamma^s(\R^n)\bigr)$
it is enough to show that
\begin{sl}
there exists positive constants $K$, $r$, $\theta$ and $\sigma$, with $\sigma>s$,
such that
\begin{equation} \label{E-estV}
\sum_{j=0}^3 |\xi|^{3-j} \, \bigl|\partial_t^j v(t;\xi)\bigr|
  \le  K \, \abs{\xi}^r \, \exp\bigl( \theta \, \abs{\xi}^{1/\sigma}\bigr) \,
       \biggl[ \sum_{j=0}^3 |\xi|^{3-j} \, \bigl|v_j(0;\xi)\bigr| \biggr] \, ,
\end{equation}
for all $t\in\zeroT$ and $\xi\in\R^n$
sufficiently large.
\end{sl}

Let
\begin{equation}\label{eq:defn.V}
V
  \defeq  \begin{pmatrix}
      (i\abs{\xi})^3 v \\
      (i\abs{\xi}) v'' \\
      (i\abs{\xi})^2 v' \\
      v'''
      \end{pmatrix}
\qquad
V_0
  \defeq  \begin{pmatrix}
      (i\abs{\xi})^3 \wh{u}_0 \\
      (i\abs{\xi}) \wh{u}_2 \\
      (i\abs{\xi})^2 \wh{u}_1 \\
      \wh{u}_3
      \end{pmatrix}
\end{equation}
so that Problem~\eqref{PdCTrasformato} is transformed
into the first order linear system
\begin{equation} \label{E-SFhom}
\begin{cases}
V'(t;\xi) = i\abs{\xi} \AA(t;\xi) \, V(t;\xi)  \\
V(0,\xi)=V_0(\xi) \, ,
\end{cases}
\end{equation}
where
\begin{equation} \label{E-A}
\AA(t;\xi)
  \defeq{}
	\left(
	\begin{array}{@{}cc|cc@{}}
	0 & 0 & 1 & 0  \\
	0 & 0 & 0 & 1  \\\hline
	0 & 1 & 0 & 0  \\
	-a & 2b & 0 & 0
	\end{array}
	\right) \, ,
\end{equation}

Note that, contrary to custom,
the second and third components in $V$ and $V_0$ are swapped.
In this way, $\AA$ assumes a favorable block structure,
making the subsequent calculations more tractable.
Note also that the eigenvalues of~$\AA$ are the characteristic roots of~the symbol~$\PP$.

\smallskip

The main tool to obtain the apriori estimate \eqref{E-estV} is
the \emph{quasi-symmetrizer} of~the matrix $\AA$.
By quasi-symmetrizer of~an $N\times N$ matrix $\AA$
we mean a family of~coercive hermitian $N\times N$ matrices~$\{\QQe \}_{0< \e \le 1}$,
such that $\QQe \AA$ is ``almost symmetric''.

In the following Proposition we gather the properties of~the quasi-symmetrizer
that we will use in the following.

\begin{Proposition} \label{P-Q}
Given $\AA(t;\xi)$ as in~\eqref{E-A},
there exists a family $\{\QQe\}_{0<\e\le1}$ of~symmetric matrices
with the following properties.
\begin{align}
&  \makebox[0pt][l]{There exists positive constants $c_1,c_2$ such that} \label{E-est Qe 1}  \\
%	anche \label{E-lbQ}
		%
		& &
		c_1 \, \e^6 \, \|V\|_{\C^4}^2
		&
		\le  \<\QQe V,V\>_{\C^4}
		\le  c_2 \, \|V\|_{\C^4}^2 \, , &
		&
		\text{for any $V\in\C^4$} \, . \notag  \\
&  \makebox[0pt][l]{There exists a positive constant $c_3$ such that} \label{E-est Qe 3}  \\
		& &
		\bigl| \< \bigl( \QQe \AA - (\QQe \AA)^* \bigr) V,V\>_{\C^4} \bigr|
		&  \le  c_3 \, \e \<\QQe V,V\>_{\C^4} \, , &
		&
		\text{for any $V\in\C^4$} \, . \notag  \\
&  \makebox[0pt][l]{If \eqref{E-HGR}--\eqref{E-HGdelta} and~\eqref{E-HGB} hold true,
		then there exists $\Lambda = \Lambda(t;\xi)$ verifying~\eqref{E-Lambda}
		such that} \label{E-est Qep} &  \\
		& &
%%%	{E-Qep VV}
		\bigl| \<\QQe' V, V\>_{\C^4} \bigr|
		&  \le  \e^{-\frac{4}{\kappa}} \, \Lambda \, \<\QQe V, V\>_{\C^4} \, ,  &
		&
		\text{for any $V\in\C^4$} \, . \notag
\end{align}
\end{Proposition}

The proof of \Cref{P-Q} constitutes the heaviest and more technical part in the proofs of \Cref{T-Gevrey-4} and \Cref{T-Gevrey-2}. 
In order not to lose the thread of the discussion,
we postpone the proof of Proposition~\ref{P-Q} to Section~\ref{S-Q}.

\medskip

Thanks to the quasi-symmetrizer $\QQe$,
we can define the \emph{approximated energy} $\EEe $.
For $V$ as in \eqref{eq:defn.V}, we let:
\begin{equation} \label{E-energy}
\EEe(t;\xi)
  \defeq {}  \< \QQe(t;\xi) V, V \>_{\C^4} \, .
\end{equation}

Differentiating~$\EEe(t;\xi)$ w.r.t.~$t$,
we get
\begin{align}
\EEe'
&  =  \< \QQe' V , V \>_{\C^4}
      +  \< \QQe V' , V \>_{\C^4}
      +  \< \QQe V , V' \>_{\C^4} \, , \notag
\intertext{%
hence, using~\eqref{E-SFhom}:}
\EEe'
&  =  \< \QQe' V , V \>_{\C^4}
      +  i\abs{\xi} \, \< \bigl( \QQe \AA-(\QQe \AA)^* \bigr)  V, V \>_{\C^4} \, . \label{E-EE2}
\end{align}

Using~\eqref{E-est Qep} to~estimate the first summand,
and~\eqref{E-est Qe 3} to~estimate the second
we get
\begin{align}
\EEe'
&	\lesssim \Bigl[ \e^{-\frac{4}{\kappa}} \, \Lambda + \e \, \abs{\xi} \Bigr] \EEe \, . \label{E-EEp le EE}
\intertext{\endgraf
We let $|\xi| \ge 1 $ and we choose $\e = \abs{\xi}^{-\frac{\kappa}{\kappa+4}}$,
so that $\e^{-\frac{4}{\kappa}}$ and $\e \, \abs{\xi}$
are balanced,
we get}
\EEe'
&  \lesssim  \abs{\xi}^{\frac{4}{\kappa+4}} \Lambda \, \EEe \, , \label{E-EEp le xi EE}
\intertext{\endgraf
By Gr\"{o}nwall Lemma and~\eqref{E-Lambda},
we get the energy estimate:}
\EEe(t;\xi)
&  \lesssim  \exp \bigl[ C \, \abs{\xi}^{\frac{4}{\kappa+4}} \bigr] \, \EEe(0;\xi) \, . \notag
\intertext{\endgraf
hence, thanks to \eqref{E-est Qe 1} we get:}
\bigl|V(t;\xi)\bigr|^2
&  \lesssim  \abs{\xi}^{\frac{6\,\kappa}{\kappa+4}} \,
  			\exp \bigl[ C \, \abs{\xi}^{\frac{4}{\kappa+4}} \bigr] \, \bigl|V(0;\xi)\bigr|^2 \, , \notag
\end{align}
which gives \eqref{E-estV}.

\subsection*{Proof of~Theorem~\ref{T-Gevrey-4} for non homogeneous equations}

If $\RR(t;\tau,\xi) \not\equiv 0$,
problem~\eqref{E-eq}--\eqref{E-IC} is transformed
into the first order linear system
\begin{equation} \label{E-SF}
\begin{cases}
V'(t;\xi) = i\abs{\xi}\AA(t;\xi) \, V(t;\xi) + \CC(t;\xi) \, V(t;\xi)  \\
V(0,\xi)=V_0(\xi) \, ,
\end{cases}
\end{equation}
where $\AA(t;\xi)$ is defined in~\eqref{E-A} and
\begin{equation} \label{E-C}
\CC(t;\xi)
  \defeq  \begin{pmatrix}
      0 & 0 & 0 & 0 \\
      0 & 0 & 0 & 0 \\
      0 & 0 & 0 & 0 \\
      c_0(t;\xi) & c_2(t;\xi) & c_1(t;\xi) & c_3(t;\xi)
      \end{pmatrix} \, ,
\end{equation}
with
\begin{align*}
c_0(t;\xi)
&  \defeq  c_{0,3}(t;\xi)
      +  c_{0,2}(t;\xi) (i\abs{\xi})^{-1}
      +  c_{0,1}(t;\xi) (i\abs{\xi})^{-2}
      +  c_{0,0}(t) (i\abs{\xi})^{-3}  \\
c_1(t;\xi)
&  \defeq  c_{1,2}(t;\xi)
      +  c_{1,1}(t;\xi) (i\abs{\xi})^{-1}
      +  c_{1,0}(t) (i\abs{\xi})^{-2}  \\
c_2(t;\xi)
&  \defeq  c_{2,1}(t;\xi)
      +  c_{2,0}(t) (i\abs{\xi})^{-1}  \\
c_3(t;\xi)
&  \defeq  c_{3,0}(t) \, ,
\end{align*}
and the $c_{j,\ell}(t;\xi)$ are defined in~\eqref{E-cjl}.

Defining the energy~$\EEe$ as in~\eqref{E-energy},
and proceding as before we get:
\begin{align*}
\EEe'
&  =  \< \QQe' V , V \>_{\C^4}
      +  \< \QQe V' , V \>_{\C^4}
      +  \< \QQe V , V' \>_{\C^4}  \\
&  =  \< \QQe' V , V \>_{\C^4}
      +  \< \QQe (i|\xi| \, \AA+\CC),V \>_{\C^4}
      +  \< \QQe V,(i|\xi| \, \AA+\CC) \>_{\C^4}  \\
&  =  \< \QQe' V , V \>_{\C^4}
      +  i\abs{\xi} \< (\QQe \AA - \AA^* \QQe ) V , V \>_{\C^4}
      +  2 \, \Re \< (\QQe \CC) V , V \>_{\C^4} \, .
\end{align*}

The first two terms can be estimated as in~the homogeneous case.
The third term is estimated with the Levi conditions:

\begin{Proposition} \label{P-Levi}
Let $\QQe$ be the symmetrizer given in Proposition~\ref{P-Q},
and let $\CC$ be the matrix in~\eqref{E-C}.
%	Let $\e = \abs{\xi}^{-\frac{\kappa}{\kappa+4}}$.
%	%
\begin{align}
&  \text{If \eqref{E-HGL c21c03}--\eqref{E-HGL c01} hold true,
		then there exists $\Lambda = \Lambda(t;\xi)$ verifying~\eqref{E-Lambda}
		such that} \label{E-est C}  \\
		&
		\bigl| \< (\QQe \CC) V , V \>_{\C^4} \bigr|
		  \le  \biggl[ \e^{-\frac{4}{k}}
		                   \sum_{j=0}^2 \Bigl( \e^{\frac{\kappa + 4}{\kappa}} \, \abs{\xi} \Bigr)^{-j}
		                +  \sum_{j=1}^3 ( \e \abs{\xi} )^{-j}
		                 \biggr]
		 \, \Lambda \, \<\QQe V, V\>_{\C^4} \, ,
		\
		\text{for any $V\in\C^4$} \, . \notag
\end{align}
\end{Proposition}

We postpone the proof of Proposition~\ref{P-Levi} to Section~\ref{S-Q}.

\smallskip

Proposition~\ref{P-Levi} shows that,
choosing $\e = \abs{\xi}^{-\frac{\kappa}{\kappa+4}}$ as before,
the Levi conditions \eqref{E-HGL c21c03}--\eqref{E-HGL c01}
ensure that the presence of~the lower-order terms does not modify the estimate~\eqref{E-EEp le xi EE},
so we proceed as in the homogeneous case.
This concludes the proof of \Cref{T-Gevrey-4}.

\subsection*{Proof of~Theorem~\ref{T-Gevrey-2}}

The proof of~Theorem~\ref{T-Gevrey-2} is analogous to that of~Theorem~\ref{T-Gevrey-4},
differing primarily in some of the estimates employed.
Indeed, if the multiplicity of~the characteristic roots of~$\PP$ is at most double,
the estimates~\eqref{E-est Qe 1}, \eqref{E-est Qep} and \eqref{E-est C} can be improved
as follows.
\begin{sl}
\begin{align}
&  \makebox[0pt][l]{There exists positive constants $c_1,c_2$ such that}  & 
	\tag{\ref{E-est Qe 1}$'$} \label{E-est Qe 1 2}  \\
		& &
		c_1 \, \e^2 \, \|V\|_{\C^4}^2
		&
		\le  \<\QQe V,V\>_{\C^4}
		\le  c_2 \, \|V\|_{\C^4}^2 \, , &
		&
		\text{for any $V\in\C^4$} \, . \notag  \\
&  \makebox[0pt][l]{If \eqref{E-HGR}--\eqref{E-HGdelta} and~\eqref{E-HGB} hold true,
		then there exists $\Lambda = \Lambda(t;\xi)$ verifying~\eqref{E-Lambda}
		such that} &
	\tag{\ref{E-est Qep}$'$} \label{E-est Qep 2}  \\
		& &
		\bigl| \<\QQe' V, V\>_{\C^4} \bigr|
		&  \le  \e^{-\frac{2}{\kappa}} \, \Lambda \, \<\QQe V, V\>_{\C^4} \, ,  &
		&
		\text{for any $V\in\C^4$} \, . \notag  \\
&  \makebox[0pt][l]{%
		If \eqref{E-HGL2 c21c03}--\eqref{E-HGL2 c03c12} hold true,
		then there exists $\Lambda = \Lambda(t;\xi)$ verifying~\eqref{E-Lambda}
		such that} 
		\tag{\ref{E-est C}$'$} \label{E-est C 2}  \\
		& &
		\bigl| \< (\QQe \CC) V , V \>_{\C^4} \bigr|
		&  \le  \Bigl[ \e^{-\frac{2}{\kappa}} + ( \e \abs{\xi} )^{-1} \Bigr] \, \Lambda \, \<\QQe V, V\>_{\C^4} \, ,  &
		&
		\text{for any $V\in\C^4$} \, . \notag
\end{align}
\end{sl}

Using~these estimates we get
\[
\EEe'
	\lesssim \Bigl[ \e^{-\frac{2}{\kappa}} \, \Lambda + \e \, \abs{\xi} + ( \e \abs{\xi} )^{-1} \Bigr] \EEe \, ,
\]
instead of~\eqref{E-EEp le EE}.
For $|\xi|\geq 1$, we choose $\e = \abs{\xi}^{-\frac{\kappa}{\kappa+2}}$,
so that $\e^{-\frac{2}{\kappa}}$ and $\e \, \abs{\xi}$
are of~the same order, while $( \e \abs{\xi} )^{-1}$ is negligible.
Thus we get
\[
\EEe'
  \lesssim  \Lambda \, \abs{\xi}^{\frac{2}{\kappa+2}} \EEe \, ,
\]
which gives the well-posedness in $\gamma^s$ with $s < 1 + \frac{\kappa}{2}$.

\section{The quasi-simmetrizer} \label{S-Q}

In this section we prove the basic properties of~the symmetrizer
and the quasi-symmetrizer of~the biwave operator.
We start by recalling the definition and the properties of~the symmetrizer of~a $2\times2$ Sylvester matrix.

\subsection{The simmetrizer of~a \texorpdfstring{$2\times2$}{2 x 2} Sylvester matrix}

In the following $\<\cdot,\cdot\>_{\C^2}$ denotes the standard inner product in $\C^2$.

\smallskip

Let
\[
\Atwo
  \defeq  \begin{pmatrix}
           0 & 1 \\
          -a & 2b
          \end{pmatrix} \, ,
\]
be a $2\times2$ matrix in Sylvester form,
and assume that
\begin{equation} \label{E-delta ge 0}
\delta
  \defeq  b^2-a \ge0 \, ,
\end{equation}
so that $\Atwo$ has real eigenvalues.

Let
\[
\Qtwo
  \defeq  \begin{pmatrix}
          \delta+b^2 & -b  \\
          -b & 1
          \end{pmatrix} \, .
\]

$\Qtwo$ is symmetric and moreover:
\[
\<\Qtwo W,W\>_{\C^2} = \delta\abs{W_1}^2 + \abs{b\,W_1-W_2}^2 \, ,
\quad
\text{for any $W = \begin{pmatrix} W_1 \\ W_2 \end{pmatrix} \in\C^2$} \, ,
\]
hence, by~\eqref{E-delta ge 0},
we see that $\Qtwo$ is positive semi-definite.

Note also that
\[
\det(\Qtwo) = \delta \, .
\]

Since
\[
\Qtwo\Atwo
  =  \begin{pmatrix} ab & -a  \\ -a & b \end{pmatrix} \, ;
\]
we see that $\Qtwo\Atwo$ is symmetric.
Thus $\Qtwo$ is a \emph{symmetrizer of~$\Atwo$}.

\subsection{The quasi-simmetrizer of~a \texorpdfstring{$2\times2$}{2 x 2} Sylvester matrix}

Let
\begin{align}
\bbe
&  \defeq  \e^2 + \sqrt{b^2+\e^4 \, } \, , \label{E-be}  \\
\aae
&  \defeq  a + \bbe \e^2 \, , \label{E-ae}  \\
\deltae
&  \defeq  \delta + \bbe \e^2 \, , \label{E-deltae}
\end{align}
so that
\[
\bbe^2-\aae = \deltae \, .
\]

The following inequalities are straightforward to verify and will be used throughout:
\begin{gather}
\abs{\bbe - b} \lesssim \e^2 \, ,
\quad
\bbe  \gtrsim  \e^2 \, ,
\quad
\aae  \gtrsim  \e^4 \, ,
\quad
\deltae  \gtrsim  \e^4 \, , \label{E-all1}  \\
\frac{\bbe }{\aae } \lesssim \e^{-2} \, ,
\quad
\frac{\bbe }{\deltae } \lesssim \e^{-2} \, ,
\quad
\frac{\bbe }{\aae \deltae }
\lesssim \e^{-6} \, . \label{E-all2}
\end{gather}

\smallskip

We consider
\begin{equation} \label{E-A2e}
\Atwo_\e
  \defeq  \begin{pmatrix}
           0 & 1 \\
          -\aae & 2\bbe
          \end{pmatrix} \, ,
\end{equation}
and its symmetrizer
\begin{equation} \label{E-Q2e}
\Qtwo_\e
  \defeq  \begin{pmatrix}
          \deltae+\bbe^2 & -\bbe  \\
          -\bbe & 1
          \end{pmatrix} \, .
\end{equation}

$\Qtwo_\e$ enjoys properties similar to those of~$\Qtwo$:

\begin{Proposition} \hfill \label{P-Qtwoe}
\begin{enumerate}
\item  \label{I-Q2se}
		$\Qtwo_\e$ is symmetric and positive definite;
		moreover:
		\[
		\<\Qtwo_\e W,W\>_{\C^2} = \deltae\abs{W_1}^2 + \abs{\bbe\,W_1-W_2}^2 \, ,
		\qquad
		\text{for any $W = \begin{pmatrix} W_1 \\ W_2 \end{pmatrix} \in\C^2$} \, .
		\]

\item   \label{I-detQ2e}
		$\det(\Qtwo_\e) = \deltae$.

\item   \label{I-Q2A2se}
		$\Qtwo_\e\Atwo_\e$ is symmetric, since
		\[
		\Qtwo_\e\Atwo_\e = \begin{pmatrix} \aae\bbe & -\aae  \\ -\aae & \bbe \end{pmatrix} \, .
		\]
\end{enumerate}
\end{Proposition}

To apply these results to the symmetrizer of~the biwave equation,
we assume that the eigenvalues of~$\Atwo$ are nonnegative.
This condition is equivalent to each of~the following:
\begin{enumerate}
\item  $a\ge0$, $b\ge0$, and $\delta \ge 0$,
\item  $\Qtwo\Atwo$ is positive semi-definite.
\end{enumerate}

\begin{Proposition}
Assume that the eigenvalues of~$\Atwo$ are nonnegative,
then
\begin{align}
\<\Qtwo_\e\Atwo_\e W,W\>_{\C^2}
&  \ge  \frac{\aae\deltae}{2\,\bbe} \, |W_1|^2
           +  \frac{\deltae}{2\,\bbe} \, |W_2|^2 \, ,  &
&  \text{for any $W = (W_1,W_2) \in \C^2$} \, . \label{E-Qe}  \\
\bigl| (\Atwo_\e - \Atwo) W \bigr|_{\C^2}
&  \lesssim  \e \<\Qtwo_\e\Atwo_\e W,W\>_{\C^2}^{\frac{1}{2}} \, ,  &
&  \text{for any $W \in \C^2$} \, . \label{E-A-Ae}
\end{align}

Let
\begin{equation} \label{E-R2e}
\Rtwoe
  \defeq  \Qtwo_\e \, ( \Atwo - \Atwo_\e ) \, ,
\end{equation}
we have
\begin{equation} \label{E-RRtwoe}
\bigl| \< \Rtwoe W , Z \>_{\C^2} \bigr|
  \lesssim  \e \Bigl[ \< \Qtwo_\e \Atwo_\e \, W ,W \>_{\C^2}
                  + \< \Qtwo_\e \, Z , Z \>_{\C^2} \Bigr]
\qquad
\text{for any $W,Z \in \C^2$} \, .
\end{equation}
\end{Proposition}

\begin{proof}[Proof of~\eqref{E-Qe}]
Let
\[
\mathbf{D}_\e
  \defeq
	\begin{pmatrix}
	\frac{\aae\deltae}{2\,\bbe} & 0 \\
	0 & \frac{\deltae}{2\,\bbe}
	\end{pmatrix}
\quad\text{and}\quad
\mathbf{S}_\e
  \defeq
	\begin{pmatrix}
	\frac{\aae\,(\bbe^2+\aae)}{2\,\bbe} & -\aae \\
	-\aae  &  \frac{\bbe^2+\aae}{2\,\bbe}
	\end{pmatrix} \, ,
\]
so that
\[
\Qtwo_\e\Atwo_\e
  =  \mathbf{D}_\e + \mathbf{S}_\e \, .
\]

As the trace and the determinant of~$\mathbf{S}_\e$ are positive,
we see that $\mathbf{S}_\e$ is positive definite,
consequently
\[
\<\Qtwo_\e\Atwo_\e W,W\>_{\C^2}
  \ge  \<\mathbf{D}_\e W,W\>_{\C^2} \, ,
\]
which gives \eqref{E-Qe}.
\end{proof}

\begin{proof}[Proof of~\eqref{E-A-Ae}]
As
\[
\Atwo_\e - \Atwo
  =  \begin{pmatrix}
     0 & 0 \\
     \aae-a & -2\,(\bbe-b)
     \end{pmatrix}
\]
we have 
\begin{equation} \label{E-A-AeZ}
\bigl| (\Atwo_\e - \Atwo) W \bigr|
  =  \bigl| \< Z , \overline{W} \>_{\C^2} \bigr| \, ,
\end{equation}
where
\[
Z
  \defeq  \begin{pmatrix} \aae-a \\ -2\,(\bbe-b) \end{pmatrix} \, .
\]

On the other side,
as $\Qtwo_\e\Atwo_\e$ is positive definite,
we have:
\begin{equation} \label{E-RZ1}
\bigl| \< Z , \overline{W} \>_{\C^2} \bigr|
  \le  \<(\Qtwo_\e\Atwo_\e)^{-1} Z,Z\>_{\C^2}^{\frac{1}{2}} \<\Qtwo_\e\Atwo_\e W,W\>_{\C^2}^{\frac{1}{2}} \, .
\end{equation}

Now, as
\[
(\Qtwo_\e\Atwo)^{-1}
  =  \frac{1}{\aae\deltae} \begin{pmatrix} \bbe  & \aae  \\ \aae  &  \aae\bbe \end{pmatrix}
\]
we have
\begin{align*}
\<(\Qtwo_\e\Atwo_\e)^{-1} Z,Z\>_{\C^2}
&  =  \frac{1}{\aae\deltae} \Bigl[\bbe\,(\aae-a)^2-4\,\aae\,(\aae-a)\,(\bbe-b)+4\,\aae\bbe(\bbe-b)^2 \Bigr] \, .
\intertext{%
As $\aae-a = \e^2 \, \bbe$, we have:}
\<(\Qtwo_\e\Atwo_\e)^{-1} Z,Z\>_{\C^2}
&  =  \frac{1}{\aae\deltae} \Bigl[\e^4\bbe^3-4\,\e^2 \aae\,\bbe\,(\bbe-b)+4\,\aae\bbe(\bbe-b)^2 \Bigr]  \\
&  =  \frac{1}{\aae\deltae}
      \Bigl[\e^4\bbe\deltae
            +\aae\bbe \bigl[ \e^2 - 2\,(\bbe-b)\bigr]^2 \Bigr] \, .
\end{align*}

Using~\eqref{E-all1} and~\eqref{E-all2},
we have
\begin{equation} \label{E-RZ1p}
\<(\Qtwo_\e\Atwo_\e)^{-1} Z,Z\>_{\C^2}
   \lesssim  \e^2 \, .
\end{equation}

Inserting \eqref{E-RZ1} and \eqref{E-RZ1p} in \eqref{E-A-AeZ},
we get \eqref{E-A-Ae}.
\end{proof}

\begin{proof}[Proof of~\eqref{E-RRtwoe}]
Using the Cauchy-Schwartz inequality for $\Qtwo_\e$ and \eqref{E-A-Ae}
we have:
\begin{align*}
\bigl| \< \Rtwoe W ,Z \>_{\C^2} \bigr|
&  =  \bigl| \< \Qtwo_\e \, ( \Atwo_\e - \Atwo ) W ,Z \>_{\C^2} \bigr|  \\
&  \le  \< \Qtwo_\e \, ( \Atwo_\e - \Atwo ) W ,( \Atwo_\e - \Atwo ) W \>_{\C^2}^{\frac{1}{2}} \,
        \< \Qtwo_\e \, Z ,Z \>_{\C^2}^{\frac{1}{2}}  \\
&  =  \bigl| ( \Atwo_\e - \Atwo ) W \bigr| \,
       \< \Qtwo_\e \, Z ,Z \>_{\C^2}^{\frac{1}{2}}  \\
&  \lesssim  \e \< \Qtwo_\e \Atwo_\e \, W ,W \>_{\C^2}^{\frac{1}{2}} \,
           \< \Qtwo_\e \, Z ,Z \>_{\C^2}^{\frac{1}{2}}  \\
&  \lesssim  \e \Bigl[ \< \Qtwo_\e \Atwo_\e \, W ,W \>_{\C^2}
                  + \< \Qtwo_\e \, Z ,Z \>_{\C^2} \Bigr] \, . \qedhere
\end{align*}
\end{proof}

\subsection{The quasi-simmetrizer of~a \texorpdfstring{$2\times2$}{2 x 2} block Sylvester matrix, general case}

Consider the matrix~$\AA$ as in \eqref{E-A}:
\[
\AA(t;\xi)
  \defeq{}
	\left(
	\begin{array}{@{}cc|cc@{}}
	0 & 0 & 1 & 0  \\
	0 & 0 & 0 & 1  \\\hline
	0 & 1 & 0 & 0  \\
	-a & 2b & 0 & 0
	\end{array}
	\right) \, .
\]

We use the previous considerations to show that the matrix
\[
\QQe
  \defeq
     \begin{pmatrix}
     \Qtwo_\e \Atwo_\e  &  \Ztwo  \\
     \Ztwo  &  \Qtwo_\e
     \end{pmatrix}
  =  \left(
	 \begin{array}{@{}cc|cc@{}}
     \aae \bbe  & -\aae  & 0 & 0  \vphantom{\bigl|}  \\
     -\aae  & \bbe  & 0 & 0  \vphantom{\bigl|}  \\\hline
     0 & 0 & \deltae +\bbe ^2 & -\bbe  \vphantom{\bigl|}  \\
     0 & 0 & -\bbe  & 1  \vphantom{\bigl|}
	\end{array}
	\right)
\]
is a quasi-symmetrizer for $\AA$,
according to Proposition~\ref{P-Q}.

\smallskip

{}From now on,
we denote by $\<\cdot,\cdot\>_{\C^4}$ the standard inner product in $\C^4$.

\smallskip

Let $V = (V_1,V_2,V_3,V_4) \in \C^4$,
we have
\[
\< \QQe V, V \>_{\C^4}
  =  \aae \bbe  \abs{V_1}^2 - \aae  V_1 \overline{V_2}
     - \aae  \overline{V_1} V_2 + \bbe  \abs{V_2}^2
     + \deltae  \abs{V_3}^2 + \abs{\bbe  V_3 - V_4}^2 \, ;
\]
gathering in~different ways the terms,
we can write
\begin{align}
\< \QQe V, V \>_{\C^4}
&  =  \frac{\aae \deltae }{\bbe } \, \abs{V_1}^2
      + \bbe \abs{V_2 - \frac{\aae }{\bbe }V_1}^2
      + \deltae  \abs{V_3}^2 + \abs{\bbe  V_3 - V_4}^2
        \label{E-Q-Ee1}  \\
\< \QQe V, V \>_{\C^4}
&  =  \frac{\aae }{\bbe } \, \abs{\bbe  V_1-V_2}^2
      + \frac{\deltae }{\bbe } \, \abs{V_2}^2
      + \deltae  \abs{V_3}^2 + \abs{\bbe  V_3 - V_4}^2 \, ,
        \label{E-Q-Ee2}
\end{align}
from which we get the following estimates
\begin{align}
\abs{V_1}^2
&  \le  \frac{\bbe }{\aae \deltae } \, \< \QQe V, V \>_{\C^4}  &
\abs{V_2 - \frac{\aae }{\bbe }V_1}^2
&  \le  \frac{1}{\bbe} \, \< \QQe V, V \>_{\C^4} \label{E-Q-e1}  \\
\abs{\bbe  V_1-V_2}^2
&  \le  \frac{\bbe }{\aae } \, \< \QQe V, V \>_{\C^4}  &
\abs{V_2}^2
&  \le  \frac{\bbe }{\deltae } \, \< \QQe V, V \>_{\C^4} \label{E-Q-e2}  \\
\abs{V_3}^2
&  \le  \frac{1}{\deltae} \, \< \QQe V, V \>_{\C^4}  &
\abs{\bbe  V_3 - V_4}^2
&  \le  \< \QQe V, V \>_{\C^4} \, . \label{E-Q-e3}
\end{align}

We are ready to prove Proposition~\ref{P-Q}.

\begin{proof}[Proof of~\eqref{E-est Qe 1}]
Using~\eqref{E-Q-e1}--\eqref{E-Q-e3}
we~have 
\begin{equation} \label{E-Qe ge V}
\begin{split}
\|V\|_{\C^4}^2
&  \le  |V_1|^2 + |V_2|^2 + (1+2 \, \bbe^2) \, |V_3|^2 + 2 \, |\bbe  V_3 - V_4|^2  \\
&  \le  \biggl[ \frac{\bbe }{\aae \deltae }
		+  \frac{\bbe }{\deltae }
		+  \frac{1+2 \, \bbe^2}{\deltae}
		+  1 \biggr] \< \QQe V, V \>_{\C^4} \, ;
\end{split}
\end{equation}
hence the lower bound in~\eqref{E-est Qe 1} follows from~\eqref{E-all2}.

The upper bound in~\eqref{E-est Qe 1} follows from the boundness of~the coefficients of~the operator.
\end{proof}

To prove \eqref{E-est Qe 3}
we introduce the following notation.

\begin{Notations}
For a vector $V=(V_1,V_2,V_3,V_4)^T \in \C^4$,
we~denote
$V^\I = (V_1,V_2)$ the vector in $\C^2$ formed by the first and second components of~$V$
and
$V^\II = (V_3,V_4)$ the vector in $\C^2$ formed by the third and fourth components of~$V$.
\end{Notations}

With the above notation, we have:
\begin{equation} \label{E-QQ-QA-Q}
\< \QQe V, V \>_{\C^4}
  =  \<\Qtwo_\e \Atwo_\e V^\I,V^\I \>_{\C^2} + \<\Qtwo_\e V^\II,V^\II \>_{\C^2} \, .
\end{equation}

\begin{proof}[Proof of~\eqref{E-est Qe 3}]
As
\begin{equation} \label{E-QQeAAe-n}
\QQe \AA
  =  \begin{pmatrix}
     \Qtwo_\e \Atwo_\e & \Ztwo \\
     \Ztwo & \Qtwo_\e
     \end{pmatrix}
     \begin{pmatrix}
     \Ztwo & \Itwo \\
     \Atwo & \Ztwo
     \end{pmatrix}
  =  \begin{pmatrix}
     \Ztwo & \Qtwo_\e \Atwo_\e  \\
     \Qtwo_\e \Atwo & \Ztwo
     \end{pmatrix} \, ,
\end{equation}
we have
\[
\QQe \AA-(\QQe \AA)^*
  =  \begin{pmatrix}
     \Ztwo & -\Rtwoe^*  \\
     \Rtwoe & \Ztwo
     \end{pmatrix}
\]
where $\Rtwoe$ is defined in \eqref{E-R2e},
hence
\[
\< \bigl( \QQe \AA - (\QQe \AA)^* \bigr) V , V \>_{\C^4}
  =  2\,\Im \< \Rtwoe V^\I ,V^\II \>_{\C^2} \, .
\]

Using \eqref{E-RRtwoe} and~\eqref{E-QQ-QA-Q}
we get~\eqref{E-est Qe 3}.
\end{proof}

\begin{proof}[Proof of~\eqref{E-est Qep}]
{}From~\eqref{E-Q-Ee2} we get
\begin{equation} \label{obiettivo}
\begin{split}
(\QQe' V,V)
&  =  \left(\frac{\aae }{\bbe }\right)'\abs{\bbe  V_1-V_2}^2
      +  \frac{2 \, \aae }{\bbe } \,
         \Re\bigl( \bbe  V_1-V_2 \, , \, \bbe 'V_1 \bigr)  \\
&  \qquad
      +  \left(\frac{\deltae }{\bbe }\right)'
         \abs{V_2}^2
      + \deltae '\abs{V_3}^2
           + 2 \, \Re\bigl( \bbe  V_3-V_4 \, , \, \bbe 'V_3 \bigr) \, .
\end{split}
\end{equation}

Thanks to~\eqref{E-Q-e1},~\eqref{E-Q-e2} and~\eqref{E-Q-e3} we have:
\begin{align*}
\left(\frac{\aae }{\bbe }\right)'\abs{\bbe  V_1-V_2}^2
&  \le  \frac{\aae }{\bbe } \,
        \left(\frac{\abs{\aae '}}{\aae }
              + \frac{\abs{\bbe '}}{\bbe }\right) \,
        \abs{\bbe  V_1-V_2}^2
   \le  \left(\frac{\abs{\aae '}}{\aae }
              + \frac{\abs{\bbe '}}{\bbe }\right) \, \<\QQe V,V\>_{\C^4}   \\
\frac{2\aae }{\bbe } \,
\Re\bigl( \bbe  V_1-V_2 \, , \, \bbe 'V_1 \bigr)
&  \le  2 \sqrt{\frac{\aae }{\bbe } \, } \, \abs{\bbe  V_1-V_2} \,
        \frac{|\bbe '|}{\sqrt{\deltae \, }} \,
        \sqrt{\frac{\aae \deltae }{\bbe } \, } \, \abs{V_1}
   \le  2 \, \frac{\abs{\bbe '}}{\sqrt{\deltae \, }} \, \<\QQe V,V\>_{\C^4}  \\
\left(\frac{\deltae }{\bbe }\right)'\abs{V_2}^2
&  \le  \frac{\deltae }{\bbe }
   \left(\frac{\abs{\deltae '}}{\deltae }
         + \frac{\abs{\bbe '}}{\bbe }\right) \, \abs{V_2}^2
   \le  \left(\frac{\abs{\deltae '}}{\deltae }
              + \frac{\abs{\bbe '}}{\bbe }\right) \, \<\QQe V,V\>_{\C^4}  \\
\deltae '\abs{V_3}^2
&   =   \frac{\deltae '}{\deltae }\deltae  \abs{V_3}^2
   \le  \frac{\abs{\deltae '}}{\deltae } \<\QQe V,V\>_{\C^4} ;  \\
2 \, \Re\bigl( \bbe  V_3-V_4 \, , \, \bbe 'V_3 \bigr)
&  \le  2\abs{\bbe  V_3-V_4} \,
        \frac{\abs{\bbe '}}{\sqrt{\deltae \, }} \,
        \sqrt{\deltae \, } \, \abs{V_3}
   \le  \frac{\abs{\bbe '}}{\sqrt{\deltae \, }} \<\QQe V,V\>_{\C^4} ,
\end{align*}

Thus, we get
\begin{equation} \label{E-QepVV}
\<\QQe' V, V\>_{\C^4}
  \lesssim \left[\frac{\abs{\aae '}}{\aae }
                 + \frac{\abs{\bbe '}}{\bbe }
                 + \frac{\abs{\deltae '}}{\deltae }
                 + \frac{\abs{\bbe '}}{\sqrt{\deltae \, }}
                  \right] \, \<\QQe V,V\>_{\C^4} \, .
\end{equation}

Now,
using~\eqref{E-be},~\eqref{E-ae} and~\eqref{E-deltae},
we show that all the terms in~the bracket
can be estimated by~$\e^{-\frac{4}{\kappa}} \, \Lambda$,
where $\Lambda=\Lambda(t;\xi)$ is a nonnegative measurable function verifying~\eqref{E-Lambda}.

For the forth term in~\eqref{E-QepVV},
recalling~\eqref{E-be}, we have
\begin{align}
\abs{\bbe '}
&   =    \Bigl|\bigl(\e^2 +\sqrt{b^2+\e^4 \, }\bigr)' \Bigr|
    =    \frac{b \, |b'|}{ \sqrt{b^2+\e^4 \, } }
   \le  |b'| \, ; \label{E-52}
\intertext{%
hence, using Hypothesis~\eqref{E-HGB}
and the fact that $\deltae\gtrsim\e^4$,
we get:}
\frac{\abs{\bbe '}}{\sqrt{\deltae \, }}
&  \lesssim  \frac{ \Lambda \, \delta^{\frac{1}{2}-\frac{1}{\kappa}}}
                  { \sqrt{\deltae \, } }
   \lesssim  \deltae^{-\frac{1}{\kappa}} \, \Lambda
   \lesssim  \e^{-\frac{4}{\kappa}} \, \Lambda \, . \label{E-11}
\end{align}

For the second term in~\eqref{E-QepVV},
as $\deltae \le \bbe^2$,
by~\eqref{E-11},
we have
\begin{equation} \label{E-12}
\frac{\abs{\bbe '}}{\bbe }
  \le  \frac{\abs{\bbe '}}{\sqrt{\deltae \, }}
   \lesssim  \e^{-\frac{4}{\kappa}} \, \Lambda \, .
\end{equation}

For the first and the third term in~\eqref{E-QepVV},
recalling~\eqref{E-ae} and~\eqref{E-deltae}, we have
\begin{align}
\frac{|\aae '|}{\aae }
&   =   \frac{\bigl|(a+ \bbe \e^2 )'\bigr|}
             {a +  \bbe \e^2 }
   \lesssim  \frac{|a'|}{a + \e^4}
        + \frac{|\bbe '|}{\bbe } \, , \label{E-13} \\
\frac{|\deltae '|}{\deltae }
&   =   \frac{\bigr|(\delta +  \bbe \e^2)'\bigr|}
             {\delta +  \bbe \e^2}
   \lesssim  \frac{|\delta'|}{\delta + \e^4}
        + \frac{|\bbe '|}{\bbe } \, . \label{E-14}
\end{align}

Using Hypothesis~\eqref{E-HGa} and~\eqref{E-HGdelta},
also the first and the third term in~\eqref{E-QepVV},
are estimated by~$\e^{-\frac{4}{\kappa}} \, \Lambda$,
and we get~ \eqref{E-est Qep}.
\end{proof}

Now we prove the Proposition~\ref{P-Levi}.

\begin{proof}[Proof of~\eqref{E-est C}]
Using the Cauchy-Schwartz inequality for $\QQe$
and the upper bound in~\eqref{E-est Qe 1}
we~have
\begin{equation} \label{E-58}
\bigl| \< \QQe C V , V \>_{\C^4} \bigr|
  \le  \< \QQe C V , C V \>_{\C^4}^{\frac{1}{2}} \, \< \QQe V , V \>_{\C^4}^{\frac{1}{2}}
  \lesssim  \|\CC V\|_{\C^4} \, \< \QQe V , V \>_{\C^4}^{\frac{1}{2}} \, .
\end{equation}

Writing
\[
c_1 \, V_3 + c_3 \, V_4
  =  (\bbe \, c_3 + c_1) \, V_3 + c_3 \, (V_4-\bbe  V_3) \, ,
\]
we have:
\begin{align*}
\|\CC V\|_{\C^4}
&  =  |c_0 \, V_1 + c_2 \, V_2 + c_1 \, V_3 + c_3 \, V_4|  \\
&  \lesssim  |c_0 \, V_1 + c_2 \, V_2| + |\bbe \, c_3 + c_1| \, |V_3| + |c_3| \, |V_4-\bbe  V_3|
   =  \I + \II + \III \, ,
\end{align*}
which we estimate using the Levi conditions.

\paragraph*{Estimate of~\texorpdfstring{$\I$}{I}}

We have
\begin{align*}
\I
  =  |c_0 \, V_1 + c_2 \, V_2|
&  \lesssim  |c_{0,3} \, V_1 + c_{2,1} \, V_2|
			+  |c_{0,2}| \, \abs{\xi}^{-1} \abs{V_1}
			+  |c_{0,1}| \, \abs{\xi}^{-2} \abs{V_1}  \\
&  \qquad  +  |c_{0,0}| \, \abs{\xi}^{-3} \abs{V_1}
			+  |c_{2,0}| \, \abs{\xi}^{-1} \abs{V_2}  \\
&   =  \I_A + \I_B + \I_C + \I_D + \I_E \, .
\end{align*}

\subparagraph*{Estimate of~$\I_A$}

As
\[
c_{0,3} \, V_1 + c_{2,1} \, V_2
  =  \frac{c_{0,3}}{\bbe} \, (\bbe V_1 - V_2)
      + \frac{c_{2,1} \, (\bbe - b)}{\bbe} \, V_2
      + \frac{c_{2,1} \, b + c_{0,3}}{\bbe} \, V_2 \, ,
\]
using~\eqref{E-Q-e2}
we have
\begin{align}
\I_A
&  \lesssim  \Biggl[ \frac{\abs{c_{0,3}}}{\sqrt{\aae \bbe \, }}
      +  \frac{|\bbe - b| \, \abs{c_{2,1}}}{\sqrt{\bbe \deltae \, }}
      +  \frac{|b \, c_{2,1} + c_{0,3}|}{\sqrt{\bbe \deltae \, }} \Biggr] \sqrt{\EEe \, } \, .
  \label{E-IA-fin}
\intertext{\endgraf
Using Hypothesis~\eqref{E-HGL c21c03},~\eqref{E-HGL c03} and~\eqref{E-HGL c21}
we have}
\I_A
&  \lesssim  \Biggl[ \frac{(ab)^{\frac{1}{2}-\frac{2}{3\kappa}}}{\sqrt{\aae \bbe \, }}
      +  \frac{|\bbe - b| \, b^{\frac{1}{2}-\frac{2}{\kappa}}}{\sqrt{\bbe \deltae \, }}
      +  \frac{(b \delta)^{\frac{1}{2}-\frac{2}{3\kappa}}}{\sqrt{\bbe \deltae \, }} \Biggr]
                \Lambda \, \sqrt{\EEe \, } \notag  \\
&  \lesssim  \Bigl[ (\aae \bbe)^{-\frac{2}{3\kappa}}
      +  \frac{|\bbe - b|}{\sqrt{\deltae \, }} \, \bbe ^{-\frac{2}{\kappa}}
      +  (\bbe \deltae)^{-\frac{2}{3\kappa}} \Bigr] \Lambda \, \sqrt{\EEe \, } \, , \notag
\end{align}
hence, thanks to~\eqref{E-all1}, we get
\begin{equation} \label{E-IA-fin2}
\I_A
  \lesssim  \e^{-\frac{4}{\kappa}} \, \Lambda \, \sqrt{\EEe \, } \, .
\end{equation}

\subparagraph*{Estimate of~$\I_B$}

Using~\eqref{E-Q-e1}, Hypothesis~\eqref{E-HGL c02}, \eqref{E-all1} and~\eqref{E-all2}
we have
\begin{align*}
\I_B
   =  |c_{0,2}| \, \abs{\xi}^{-1} \abs{V_1}
&  \lesssim  |c_{0,2}| \, \abs{\xi}^{-1} \, \Bigl(\frac{\bbe }{\aae \deltae }\Bigr)^{\frac{1}{2}} \, \sqrt{\EEe\,}  \\
&  \lesssim  b^{1-\frac{4}{\kappa}} \, \abs{\xi}^{-1} \,
            \Bigl(\frac{\bbe }{\aae \deltae }\Bigr)^{\frac{1}{2}} \, \Lambda \, \sqrt{\EEe\,}  \\
&  \le   \Bigl(\frac{\bbe}{\aae}\Bigr)^{1-\frac{4}{\kappa}} \,
             \frac{1}{\aae^{-\frac{1}{2}+\frac{4}{\kappa}}} \,
             \Bigl(\frac{\bbe }{\deltae }\Bigr)^{\frac{1}{2}} \, \abs{\xi}^{-1} \Lambda \, \sqrt{\EEe\,}  \\
&  \lesssim  \e^{-2\,(1-\frac{4}{\kappa})-4\,(-\frac{1}{2}+\frac{4}{\kappa})-1} \,
             \abs{\xi}^{-1} \, \Lambda \, \sqrt{\EEe\,}
    =  \e^{-1 - \frac{8}{\kappa}} \, \abs{\xi}^{-1} \, \Lambda \, \sqrt{\EEe\,} \, .
\end{align*}

\subparagraph*{Estimate of~$\I_C$}

Using~\eqref{E-Q-e1}, Hypothesis~\eqref{E-HGL c01}, \eqref{E-all1} and~\eqref{E-all2}
we have
\begin{align*}
\I_C
  =  |c_{0,1}| \, \abs{\xi}^{-2} \abs{V_1}
&  \lesssim  |c_{0,1}| \, \abs{\xi}^{-2} \, \Bigl(\frac{\bbe }{\aae \deltae }\Bigr)^{\frac{1}{2}} \, \sqrt{\EEe\,}  \\
&  \le  \Lambda \, b^{\frac{1}{2}-\frac{6}{\kappa}} \,
            \abs{\xi}^{-2} \,
            \Bigl(\frac{\bbe }{\aae \deltae }\Bigr)^{\frac{1}{2}} \, \sqrt{\EEe\,}  \\
&   =   \Bigl(\frac{\bbe}{\aae}\Bigr)^{\frac{1}{2}-\frac{6}{\kappa}} \,
        \frac{1}{\aae^{\frac{6}{\kappa}}} \,
        \Bigl(\frac{\bbe }{\deltae }\Bigr)^{\frac{1}{2}} \,
        \abs{\xi}^{-2} \, \Lambda \, \sqrt{\EEe\,}  \\
&  \lesssim  \e^{-2 (\frac{1}{2}-\frac{6}{\kappa})-\frac{24}{\kappa}-1} \, \abs{\xi}^{-2} \, \Lambda
      =      \e^{-2-\frac{12}{\kappa}} \, \abs{\xi}^{-2} \, \Lambda \, \sqrt{\EEe\,} \, .
\end{align*}

\subparagraph*{Estimate of~$\I_D$}

As $c_{0,0}\in L^1\bigl(\zeroT\bigr)$,
using~\eqref{E-Q-e1} and~\eqref{E-all2}
we have
\begin{align*}
\I_D
   =  |c_{0,0}| \, \abs{\xi}^{-3} \abs{V_1}
  \lesssim  \abs{\xi}^{-3} \Bigl(\frac{\bbe }{\aae \deltae }\Bigr)^{\frac{1}{2}} \, \Lambda \, \sqrt{\EEe \, }
  \lesssim  ( \e \abs{\xi} )^{-3} \, \Lambda \, \sqrt{\EEe \, } \, .
\end{align*}

\subparagraph*{Estimate of~$\I_E$}

As $c_{2,0}\in L^1\bigl(\zeroT\bigr)$,
using~\eqref{E-Q-e2} and~\eqref{E-all2}
we have
\[
\I_E
  =  |c_{2,0}| \, \abs{\xi}^{-1} \abs{V_2}
  \lesssim  \abs{\xi}^{-1} \, \sqrt{ \frac{\bbe }{\deltae } \, } \, \Lambda \, \sqrt{\EEe \, }
  \lesssim  ( \e \abs{\xi} )^{-1} \, \Lambda \, \sqrt{\EEe \, } \, .
\]

Gathering the above estimates, we have:
\[
\I
  \lesssim  \Bigl[ \e^{-\frac{4}{k}}
                +  \e^{-1-\frac{8}{\kappa}} \, \abs{\xi}^{-1}
                +  \e^{-2-\frac{12}{\kappa}} \, \abs{\xi}^{-2}
                + ( \e \abs{\xi} )^{-1}
                + ( \e \abs{\xi} )^{-3}
                 \Bigr] \,
         \Lambda \, \sqrt{\EEe \, } \, .
\]

\paragraph*{Estimate of~\texorpdfstring{$\II$}{II}}

Using~\eqref{E-Q-e3},
we have
\[
|\bbe \, c_3 + c_1| \, |V_3|
  \lesssim  \frac{|\bbe \, c_3 + c_1|}{\sqrt{\deltae \, }} \, \sqrt{\EEe \, } \, ,
\]
and moreover
\[
|\bbe \, c_3 + c_1|
  \lesssim  |b \, c_{3,0} + c_{1,2}| + |\bbe - b| \, |c_{3,0}|
            +  |c_{1,1}| \abs{\xi}^{-1} + |c_{1,0}| \abs{\xi}^{-2} \, .
\]

Using Hypothesis~\eqref{E-HGL c03c12} and~\eqref{E-all1}
we have
\[
\frac{|b \, c_{3,0} + c_{1,2}|}{\sqrt{\deltae \, }}
  \lesssim  \frac{\delta^{\frac{1}{2}-\frac{1}{k}}}{\sqrt{\deltae \, }} \, \Lambda
  \lesssim  \deltae^{-\frac{1}{k}} \, \Lambda
  \lesssim  \e^{-\frac{4}{k}} \, \Lambda \, .
\]

Using~\eqref{E-all1}, we have
\[
\frac{|\bbe - b| \, |c_{3,0}|}{\sqrt{\deltae \, }}
  \lesssim  \frac{\e^2}{\e^2}  =  1 \, .
\]

Using Hypothesis~\eqref{E-HGL c11}, \eqref{E-all1} and~\eqref{E-all2}
we have
\begin{align*}
\frac{|c_{1,1}| \abs{\xi}^{-1}}{\sqrt{\deltae \, }}
&  \le       \Lambda \, \frac{b^{\frac{1}{2}-\frac{4}{\kappa}} \, \abs{\xi}^{-1}}{\sqrt{\deltae \, }}
   \le       \Lambda \, \Bigl(\frac{\bbe}{\deltae}\Bigr)^{\frac{1}{2}-\frac{4}{\kappa}} \,
            \frac{1}{\deltae^{\frac{4}{\kappa}}} \, \abs{\xi}^{-1}  \\
&  \lesssim  \e^{-2(\frac{1}{2}-\frac{4}{\kappa})-\frac{16}{\kappa}} \, \abs{\xi}^{-1} \, \Lambda
   \lesssim  \e^{-1-\frac{8}{\kappa}} \, \abs{\xi}^{-1} \, \Lambda \, .
\end{align*}

As $c_{1,0} \in L^1\bigl(\zeroT\bigr)$,
we have
\[
\frac{|c_{1,0}| \abs{\xi}^{-2}}{\sqrt{\deltae \, }}
  \lesssim  (\e \abs{\xi})^{-2} \, \Lambda \, .
\]

Gathering the above estimates, we have:
\[
\II
  \lesssim  \bigl[ e^{-\frac{4}{k}} + \e^{-1-\frac{8}{\kappa}} \, \abs{\xi}^{-1}
                   + ( \e \abs{\xi} )^{-2} \bigr] \,
            \Lambda \, \sqrt{\EEe \, } \, .
\]

\paragraph*{Estimate of~\texorpdfstring{$\III$}{III}}

As $c_{3,0} \in L^1\bigl(\zeroT\bigr)$,
we have
\[
\III
  \lesssim  \Lambda \, \sqrt{\EEe \, } \, .
\]

\paragraph*{Conclusion of~the proof}

Gathering the estimates of~$\I$, $\II$ and~$\III$ in~\eqref{E-58}
we get~\eqref{E-est C}.
\end{proof}

\subsection{The quasi-simmetrizer of~a $2\times2$ block Sylvester matrix,
	with at most double multiplicity}

As $b(t;\xi) \ge \beta_0$,
with $\beta_0>0$,
we have
\begin{equation} \label{E-all 2}
\aae  \gtrsim  \e^2 \, ,
\quad
\deltae  \gtrsim  \e^2 \, ,
\quad
\frac{\bbe }{\aae \deltae }  \gtrsim  \e^{-2} \, ,
\end{equation}
instead of~\eqref{E-all1} and~\eqref{E-all2}.
The first two inequalities follows from the definition of~$\aae$ and $\deltae$.
The last inequality follows from the identity
\[
\frac{\bbe }{\aae \deltae }
  =  \frac{1}{\bbe } \Bigl( \frac{1}{\aae } + \frac{1}{\deltae } \Bigr) \, .
\]

The proof of~\eqref{E-est Qe 1 2} and~\eqref{E-est C 2} follows the same lines
as those of~\eqref{E-est Qe 1} and~\eqref{E-est C},
with the only modification being the use of~estimate~\eqref{E-all 2}
in place of~estimate~\eqref{E-all1} and~\eqref{E-all2}.

\begin{proof}[Proof of~\eqref{E-est C 2}]
As in the proof of~\eqref{E-est C},
it's enough to estimate $\|\CC \, V\|_{\C^4}$.

Let $\CC_0(t;\xi)$ be the homogeneous part of~order $0$ of~$\CC(t;\xi)$:
\[
\CC_0(t;\xi)
  \defeq  \begin{pmatrix}
	      0 & 0 & 0 & 0 \\
	      0 & 0 & 0 & 0 \\
	      0 & 0 & 0 & 0 \\
	      c_{0,3}(t;\xi) & c_{2,1}(t;\xi) & c_{1,2}(t;\xi) & c_{3,0}(t;\xi)
	      \end{pmatrix} \, ,
\]
and let
\[
\CC_1(t;\xi)
  \defeq  \CC(t;\xi) - \CC_0(t;\xi) \, ,
\]
so that
\[
\|\CC \, V\|_{\C^4}
  \le  \|\CC_0 \, V\|_{\C^4} + \|\CC_1 \, V\|_{\C^4} \, .
\]

We estimate $\mathscr{C}_0$ via the Levi conditions,
while the lower-order terms in $\mathscr{C}_1$ are controlled directly by the energy,
owing to the fact that the characteristic roots are at most double.

\smallskip

We have:
\begin{align*}
\|\CC_0 V\|_{\C^4}
&  =  |c_{0,3} \, V_1 + c_{2,1} \, V_2 + c_{2,1} \, V_3 + c_{3,0} \, V_4|  \\
&  \lesssim  |c_{0,3} \, V_1 + c_{2,1} \, V_2| + |\bbe \, c_{3,0} + c_{1,2}| \, |V_3| + |c_{3,0}| \, |V_4-\bbe  V_3|
   =  \I + \II + \III \, ,
\end{align*}

\subparagraph*{Estimate of~$\I$}

As
\begin{align*}
c_{0,3} \, V_1 + c_{2,1} \, V_2
&  =  \frac{c_{0,3}}{\bbe} \, (\bbe V_1 - V_2)
      + \frac{c_{2,1} \, (\bbe - b)}{\bbe} \, V_2
      + \frac{c_{2,1} \, b + c_{0,3}}{\bbe} \, V_2 \, ,
\intertext{%
using~\eqref{E-Q-e2}
we have}
\I
&  \lesssim  \Biggl[ \frac{\abs{c_{0,3}}}{\sqrt{\aae \bbe \, }}
      +  \frac{|\bbe - b| \, \abs{c_{2,1}}}{\sqrt{\bbe \deltae \, }}
      +  \frac{|b \, c_{2,1} + c_{0,3}|}{\sqrt{\bbe \deltae \, }} \Biggr] \sqrt{\EEe \, } \, .
\intertext{\endgraf
The second term is bounded uniformly w.r.t. $\e$.
Using Hypothesis~\eqref{E-HGL2 c21c03} and~\eqref{E-HGL2 c03},
to estimate the first and third term, we get}
\I
&  \lesssim  \Biggl[ \frac{a^{\frac{1}{2}-\frac{1}{\kappa}}}{\sqrt{\aae \bbe \, }}
      +  \frac{|\bbe - b| \, \abs{c_{2,1}}}{\sqrt{\bbe \deltae \, }}
      +  \frac{\delta^{\frac{1}{2}-\frac{1}{\kappa}}}{\sqrt{\bbe \deltae \, }} \Biggr]
                \Lambda \, \sqrt{\EEe \, }  \\
&  \lesssim  \Bigl[ \aae ^{-\frac{1}{\kappa}}
      +  \e^2 \deltae ^{-\frac{1}{2}}
      +  \deltae ^{-\frac{1}{\kappa}} \Bigr] \Lambda \, \sqrt{\EEe \, }
   \lesssim  \e^{-\frac{2}{\kappa}} \, \Lambda \, \sqrt{\EEe \, } \, .
\end{align*}

\subparagraph*{Estimate of~$\II$}

We have
\[
|\bbe \, c_{3,0} + c_{1,2}| \, |V_3|
  \lesssim  \frac{|b \, c_{3,0} + c_{1,2}| + |\bbe - b| \, |c_{3,0}|}
                 {\sqrt{\deltae \, }} \, \sqrt{\EEe \, } \, .
\]

Using Hypothesis~\eqref{E-HGL c03c12}
we have
\[
\frac{|b \, c_{3,0} + c_{1,2}|}{\sqrt{\deltae \, }}
  \lesssim  \frac{\delta^{\frac{1}{2}-\frac{1}{k}}}{\sqrt{\deltae \, }} \, \Lambda
  \lesssim  \deltae^{-\frac{1}{k}} \, \Lambda
  \lesssim  \e^{-\frac{2}{k}} \, \Lambda \, ,
\]
whereas the second term is bounded uniformly w.r.t. $\e$.
Thus, we get:
\[
\II
  \lesssim  \e^{-\frac{2}{\kappa}} \, \Lambda \, \sqrt{\EEe \, } \, .
\]

\subparagraph*{Estimate of~$\III$}

As $c_{3,0} \in L^1\bigl(\zeroT\bigr)$,
using~\eqref{E-Q-e3},
we get
\[
\III
  \lesssim  \Lambda \, \sqrt{\EEe \, } \, .
\]

\paragraph*{Estimate of~$\|\mathscr{C}_0 V\|_{\C^4}$}

Gathering the estimates of~$\I$, $\II$ and~$\III$,
we get
\[
\|\mathscr{C}_0 V\|_{\C^4}
  \lesssim  \e^{-\frac{2}{k}} \, \Lambda \, \sqrt{\EEe \, } \, .
\]

\paragraph*{Estimate of~$\|\mathscr{C}_1 V\|_{\C^4}$}

We have
\[
\|\mathscr{C}_1 V\|_{\C^4}
  \lesssim  \Bigl[ \abs{\xi}^{-1} \, \|V\|_{\C^4} + \abs{\xi}^{-2} \, \|V\|_{\C^4} + \abs{\xi}^{-3} \, \|V\|_{\C^4} \Bigr] \Lambda \, ,
\]
and, since it's enough to prove the energy estimate only for $|\xi|\ge1$ (cf.~\eqref{E-estV}),
we have
\[
\|\mathscr{C}_1 V\|_{\C^4}
  \lesssim  \abs{\xi}^{-1} \, \Lambda  \, \|V\|_{\C^4} \, ,
\qquad
\text{for $|\xi|\ge1$} \, .
\]

Using \eqref{E-est Qe 1 2} we have
\[
\|\mathscr{C}_1 V\|_{\C^4}
  \lesssim  ( \e \abs{\xi} )^{-1} \, \Lambda \, \sqrt{\EEe \, } \, .
\]

The estimates for $\|\mathscr{C}_0 V\|_{\C^4}$ and $\|\mathscr{C}_1 V\|_{\C^4}$
yelds the estimate for $\|\mathscr{C} V\|_{\C^4}$,
and~\eqref{E-est C 2} follows.
\end{proof}

\appendix
\section{Operators with constant coefficients} \label{A-1}

Various conditions equivalent to the classical G\aa rding condition~\eqref{E-Garding}
have been found (see~\cite[Theorem~12.4.6]{Hormander},~\cite[Theorem~2.1]{Svensson-1969}),
we use the~formulation given by~Peyser~\cite{Peyser-1963a}.
In order to~state this condition,
we give the following definition:

\begin{Definition} \label{D-pd}
Let
\[
\PP(\tau,\xi) \equiv  \bigl(\tau-\tau_1(\xi)\bigr) \cdots \bigl(\tau-\tau_m(\xi)\bigr)
\]
be a hyperbolic polynomial of~degree $m\ge2$,
whose coefficients depend on~$\xi\in\R^n$.
We~say that a polynomial $\RR(\tau,\xi)$ of~degree~$\le m-1$
has a~\emph{proper decomposition w.r.t.~$\PP(\tau,\xi)$}
if~there exist some functions $\ell_k \in L^\infty(\R^n)$ such that
\begin{equation} \label{E-dec ell 0}
\RR(\tau,\xi)
  =  \sum_{k=1}^m \ell_k(\xi) \, \PP_{\wh{k}}(\tau,\xi) \, ,
\qquad
\text{for all~$\xi\in\R^n$} \, ,
\end{equation}
where
\begin{equation} \label{E-reduced}
\PP_{\wh{k}}(\tau,\xi)
  \equiv  \prod_{\substack{j=1,\dotsc,m \\ j\ne k}}
          \bigl(\tau-\tau_j(\xi)\bigr) \, ,
\qquad k=1,\dotsc,m.
\end{equation}
are the~so called \emph{reduced} (or~\emph{incomplete}) polynomials of~$\PP(\tau,\xi)$.
\end{Definition}

\begin{Example} \label{Ex-dtPP}
As
\begin{equation} \label{E-Pp Ph}
\partial_\tau \PP(\tau,\xi)
  =  \sum_{h=1}^m \PP_{\wh{h}}(\tau,\xi) \, ,
\end{equation}
the polynomial $\partial_\tau \PP(\tau,\xi)$
has a proper decomposition w.r.t.~$\PP(\tau,\xi)$.
\end{Example}

According to~\cite[Theorem~2.1]{Svensson-1969},
\eqref{E-Garding} is equivalent to the following condition
given by~Peyser \cite{Peyser-1963a}:
\begin{sl}
\begin{equation} \label{E-PD}
\text{For each $d=1,\dotsc,m-1$
    $\RR_d(\tau,\xi)$ has a proper decomposition
    w.r.t.~$\partial_\tau^{m-1-d} \PP(\tau,\xi)$.}
\end{equation}
\end{sl}
Condition~\eqref{E-PD} means that each $\RR_d(\tau,\xi)$
is a linear combination, with bounded coefficients, of~the reduced polynomials of~$\partial_\tau^{m-1-d} \PP(\tau,\xi)$.

\smallskip

In this appendix we explicit the conditions of~proper decomposition
w.r.t.~a biquadratic polynomial
\begin{equation} \label{E-P}
\PP(\tau,\xi)
  =  \tau^4 - 2 \, b(\xi) \, \tau^2 + a(\xi) \, .
\end{equation}

\begin{Notations}
In the following, given two continuous functions $\varphi(\xi)$ and $\psi(\xi)$ with
\[
\bigl|\varphi(\xi)\bigr|
  \lesssim  \psi(\xi) \, ,
\]
we denote
\[
\Dfrac{\varphi(\xi)}{\psi(\xi)}
  \defeq
  \begin{cases}
  \frac{\varphi(\xi)}{\psi(\xi)}  &  \text{if $\psi(\xi) \ne 0$} \, ,  \\*[3pt]
  0  &  \text{if $\psi(\xi) = 0$} \, ,
  \end{cases}
\]
so that
\[
\varphi(\xi)
  =  \Dfrac{\varphi(\xi)}{\psi(\xi)} \, \psi(\xi) \, ,
\]
\text{for any $\xi\in\R^n$} \, .
\end{Notations}

\subsection{Proper decomposition of~$\RR_1(\tau,\xi)$ w.r.t.~$\partial_{\tau\tau}^2 \PP(\tau,\xi)$}

\begin{Lemma}
The polynomial
\[
\RR_1(\tau,\xi)
  =  c_{1,0}(\xi) \, \tau + c_{0,1}(\xi)
\]
has a proper decomposition w.r.t.~$\partial_{\tau\tau}^2 \PP(\tau,\xi)$ if, and only if,
\begin{equation} \label{E-pd1}
\bigl|c_{0,1}(\xi)\bigr|
  \lesssim  \sqrt{b(\xi) \, } \, .
\end{equation}
\end{Lemma}

\begin{proof}
As
\[
\partial_{\tau\tau}^2 \PP(\tau,\xi)
  =  12 \, \tau^2 - 4 \, b(\xi)
  =  12 \, \Bigl[\tau-\sqrt{b(\xi)/3 \, }\Bigr] \, \Bigl[\tau+\sqrt{b(\xi)/3 \, }\Bigr] \, ,
\]
$\RR_1$ has a proper decomposition w.r.t.~$\partial_{\tau\tau}^2 \PP(\tau,\xi)$
if, and only if,
\begin{equation} \label{E-propdec1}
c_{1,0}(\xi) \, \tau + c_{0,1}(\xi)
  =  \ell_1(\xi) \, \Bigl[\tau-\sqrt{b(\xi)/3 \, }\Bigr]
     +  \ell_2(\xi) \, \Bigl[\tau+\sqrt{b(\xi)/3 \, }\Bigr] \, ,
\end{equation}
for some $\ell_1,\ell_2 \in L^\infty$.

Setting $\tau=\sqrt{b(\xi)/3 \, }$ in~\eqref{E-propdec1},
we see that the boundness of~$\ell_2$ implies~\eqref{E-pd1}.

As
\[
\begin{split}
c_{1,0}(\xi) \, \tau + c_{0,1}(\xi)
&  =  \frac{1}{2} \, \Biggl[ c_{1,0}(\xi) - \Dfrac{c_{0,1}(\xi)}{\sqrt{b(\xi)/3 \, }} \Biggr] \,
                      \Bigl[ \tau-\sqrt{b(\xi)/3 \, } \Bigr]  \\
&     \qquad  +  \frac{1}{2} \, \Biggl[ c_{1,0}(\xi) + \Dfrac{c_{0,1}(\xi)}{\sqrt{b(\xi)/3 \, }} \Biggr] \,
                       \Bigl[ \tau+\sqrt{b(\xi)/3 \, } \Bigr] \, ,
\end{split}
\]
we see that if~\eqref{E-pd1} is satisfied
then~\eqref{E-propdec1} holds true.
\end{proof}

\subsection{Proper decomposition of~$\RR_2(\tau,\xi)$ w.r.t.~$\partial_\tau \PP(\tau,\xi)$}

\begin{Lemma} \label{L-R2}
The polynomial
\[
\RR_2(\tau,\xi)
  =  c_{2,0}(\xi) \, \tau^2 + c_{1,1}(\xi) \, \tau + c_{0,2}(\xi)
\]
has a proper decomposition w.r.t.~$\partial_\tau \PP(\tau,\xi)$ if, and only if,
\begin{subequations}
\begin{align}
\bigl|c_{0,2}(\xi)\bigr|
&  \lesssim  b(\xi) \label{E-pd21}  \\
\bigl|c_{1,1}(\xi)\bigr|
&  \lesssim  \sqrt{b(\xi) \, } \, . \label{E-pd22}
\end{align}
\end{subequations}
\end{Lemma}

\begin{proof}
As
\[
\partial_\tau \PP(\tau)
  =  4 \, \tau^3 - 4 \, b(\xi) \, \tau
  =  4 \, \tau \, \Bigl[ \tau-\sqrt{b(\xi) \, } \Bigr] \,
              \Bigl[ \tau+\sqrt{b(\xi) \, } \Bigr] \, ,
\]
$\RR_2(\tau,\xi)$ has a proper decomposition w.r.t.~$\partial_\tau \PP(\tau,\xi)$ if, and only if,
\begin{equation} \label{E-propdec2}
\begin{split}
c_{2,0}(\xi) \, \tau^2 +  &  c_{1,1}(\xi) \, \tau + c_{0,2}(\xi)  \\
&  =  \ell_1(\xi) \, \tau \, \Bigl[ \tau-\sqrt{b(\xi) \, } \Bigr]
     +  \ell_2(\xi) \, \tau \, \Bigl[ \tau+\sqrt{b(\xi) \, } \Bigr]
     +  \ell_3(\xi) \, \bigl[ \tau^2-b(\xi) \bigr]
\end{split}
\end{equation}
for some $\ell_1,\ell_2,\ell_3 \in L^\infty$.

Setting $\tau=0$ in~\eqref{E-propdec2},
we see that the boundness of~$\ell_3(\xi) $
implies~\eqref{E-pd21}.

If $b(\xi)\equiv0$ then~\eqref{E-propdec2}
implies that $c_{1,1}(\xi)\equiv0 $ and $c_{0,2}(\xi)\equiv0$,
whereas if $b(\xi) \not\equiv0$,
setting $\tau=\sqrt{b(\xi) \, }$ in~\eqref{E-propdec2},
we see that the boundness of~$\ell_2(\xi)$
implies~\eqref{E-pd22}.

On the other side
writing
\begin{align*}
c_{2,0}(\xi) \, \tau^2 + c_{1,1}(\xi) \, \tau + c_{0,2}(\xi)
&  =  \frac{1}{2} \,
      \Bigl(c_{2,0}(\xi) - \Dfrac{c_{1,1}(\xi)}{\sqrt{b(\xi) \, }} + \Dfrac{c_{0,2}(\xi)}{b(\xi)}\Bigr) \,
      \tau \, \Bigl[ \tau - \sqrt{b(\xi) \, } \Bigr]  \\
&  \qquad
   +  \frac{1}{2} \,
      \Bigl(c_{2,0}(\xi) + \Dfrac{c_{1,1}(\xi)}{\sqrt{b(\xi) \, }} + \Dfrac{c_{0,2}(\xi)}{b(\xi)}\Bigr) \,
      \tau \, \Bigl[ \tau + \sqrt{b(\xi) \, } \Bigr]  \\
&  \qquad
   -  \Dfrac{c_{0,2}(\xi)}{b(\xi)} \, \bigl[ \tau^2 - b(\xi) \bigr] \, ,
\end{align*}
we see that if~\eqref{E-pd21} and~\eqref{E-pd22} are satisfied
then~\eqref{E-propdec2} holds true.
\end{proof}

\subsection{Proper decomposition of~$\RR_3(\tau,\xi)$ w.r.t.~$\PP(\tau,\xi)$}

\begin{Lemma}
The polynomial
\[
\RR_3(\tau,\xi)
  =  c_{3,0}(\xi) \, \tau^3 + c_{2,1}(\xi) \, \tau^2 + c_{1,2}(\xi) \, \tau + c_{0,3}(\xi)
\]
has a proper decomposition w.r.t.~$\PP(\tau,\xi)$ in~\eqref{E-P}
if, and only if,
\begin{align}
\bigl[b(\xi) \, c_{3,0}(\xi) + c_{1,2}(\xi)\bigr]^2
&  \lesssim  \delta(\xi) \label{E-Lcc1}  \\
\bigl[b(\xi) \, c_{2,1}(\xi) + c_{0,3}(\xi)\bigr]^2
&  \lesssim  b(\xi) \, \delta(\xi) \label{E-Lcc2}  \\
\bigl[c_{0,3}(\xi)\bigr]^2
&  \lesssim  a(\xi) \, b(\xi) \label{E-Lcc3}  \\
\bigl[c_{2,1}(\xi)\bigr]^2
&  \lesssim  b(\xi) \label{E-Lcc4} \, .
\end{align}
\end{Lemma}

\begin{proof}
We denote by $\pm\tau_1(\xi)$ and $\pm\tau_2(\xi)$ the roots
of the polynomial~$\tau \mapsto \PP(\tau,\xi)$,
with
\[
\bigl((\tau_1(\xi)\bigr)^2
  =  b(\xi) - \sqrt{\delta(\xi)\,} \, ,
\qquad
\bigl(\tau_2(\xi)\bigr)^2 = b(\xi) + \sqrt{\delta(\xi)\,} \, .
\]

By definition, $\RR_3(\tau,\xi)$ has a proper decomposition w.r.t.~$\PP(\tau,\xi)$ if, and only if,
\begin{multline} \label{E-propdec3}
c_{3,0}(\xi) \, \tau^3 + c_{2,1}(\xi) \, \tau^2 + c_{1,2}(\xi) \, \tau + c_{0,3}(\xi)  \\
  =  \ell_2(\xi) \, \PPP{1}{+}(\tau,\xi) + \ell_1(\xi) \, \PPP{1}{-}(\tau,\xi)
     + \ell_4(\xi) \, \PPP{2}{+}(\tau,\xi) + \ell_3(\xi) \, \PPP{2}{-}(\tau,\xi)
\end{multline}
where
\begin{align*}
\PPP{1}{+}(\tau,\xi)
&  \defeq  \frac{\PP(\tau,\xi)}{\tau-\tau_1(\xi)}
    =  \bigl(\tau+\tau_1(\xi)\bigr)\bigl(\tau^2-\tau_2^2(\xi)\bigr)  \\
\PPP{1}{-}(\tau,\xi)
&  \defeq  \frac{\PP(\tau,\xi)}{\tau+\tau_1(\xi)}
    =  \bigl(\tau-\tau_1(\xi)\bigr)\bigl(\tau^2-\tau_2^2(\xi)\bigr)  \\
\PPP{2}{+}(\tau,\xi)
&  \defeq  \frac{\PP(\tau,\xi)}{\tau-\tau_2(\xi)}
    =  \bigl(\tau+\tau_2(\xi)\bigr)\bigl(\tau^2-\tau_1^2(\xi)\bigr)  \\
\PPP{2}{-}(\tau,\xi)
&  \defeq  \frac{\PP(\tau,\xi)}{\tau+\tau_2(\xi)}
    =  \bigl(\tau-\tau_2(\xi)\bigr)\bigl(\tau^2-\tau_1^2(\xi)\bigr)
\end{align*}
are the reduced polynomials of~$\PP(\tau,\xi)$,
and $\ell_1,\ell_2,\ell_3,\ell_4 \in L^\infty$.

We have to split $\R^n$
into four zones:
\begin{align*}
\Omega_1
&  =  \Bigl\{ \ \xi \in \R^n \ \Bigm|
              \ a(\xi) > 0 \text{ and } \delta(\xi) > 0 \ \Bigr\}  \\
\Omega_2
&  =  \Bigl\{ \ \xi \in \R^n \ \Bigm|
              \ a(\xi) > 0 \text{ and } \delta(\xi) = 0 \ \Bigr\}  \\
\Omega_3
&  =  \Bigl\{ \ \xi \in \R^n \ \Bigm|
              \ a(\xi) = 0 \text{ and } \delta(\xi) > 0 \ \Bigr\}  \\
\Omega_4
&  =  \Bigl\{ \ \xi \in \R^n \ \Bigm|
              \ a(\xi) = 0 \text{ and } \delta(\xi) = 0 \ \Bigr\} \, .
\end{align*}

It's clear that the functions $\ell_j$s in~\eqref{E-propdec3}
are bounded in~$\zeroT\times\R^n$ if, and only if,
they are bounded in~each~$\Omega_j$, for $j=1,2,3,4$;
analogously estimates~\eqref{E-Lcc1}--\eqref{E-Lcc4} hold true
in~$\zeroT\times\R^n$ if, and only if,
they hold in~each~$\Omega_j$, for $j=1,2,3,4$.

\smallskip

Note that if $b(\xi)>0$, that is if~$\xi\in\Omega_1\cup\Omega_2\cup\Omega_3$,
condition~\eqref{E-Lcc4} is a consequence of~\eqref{E-Lcc2} and~\eqref{E-Lcc3},
since
\begin{align*}
c_{2,1}^2(\xi)
&  =  \frac{\bigl[b(\xi) \, c_{2,1}(\xi)\bigr]^2}{b^2(\xi)}  \\
&  \lesssim  \frac{\bigl[b(\xi) \, c_{2,1}(\xi) + c_{0,3}(\xi)\bigr]^2 + c_{0,3}^2(\xi)}{b^2(\xi)}  \\
&  \lesssim  \frac{b(\xi) \, \delta(\xi) + a(\xi) \, b(\xi)}{b^2(\xi)}
   =  b(\xi) \, .
\end{align*}

However as the previous calculation can be carried out only if $b(\xi)>0$,
we cannot omit~\eqref{E-Lcc4} since~\eqref{E-Lcc2} and~\eqref{E-Lcc3}
will not ensure that $c_{2,1}(\xi)$ vanishes if $b(\xi)$ does.

\subsubsection*{Case~I. $\xi \in \Omega_1$: $a(\xi) > 0$ and $\delta(\xi) > 0$}

In this case the roots of~$\PP(\tau,\xi)$ are
different from each other
and different from $0$.
Thus if we set in~\eqref{E-propdec3}
$\tau=\tau_1$,
$\tau=-\tau_1$,
$\tau=\tau_2$,
$\tau=-\tau_2$
we get
\begin{align*}
\RR_3(\tau_1)
&  =  \ell_1(\xi) \, \PPP{1}{+}(\tau_1) \, , &
\RR_3(-\tau_1)
&  =  \ell_2(\xi) \, \PPP{1}{-}(-\tau_1) \, ,  \\
\RR_3(\tau_2)
&  =  \ell_3(\xi) \, \PPP{2}{+}(\tau_2) \, , &
\RR_3(-\tau_2)
&  =  \ell_4(\xi) \, \PPP{2}{-}(-\tau_2)
\end{align*}
and the left hand sides are different from zero for any~$\xi\in\Omega_1$.

We have
\[
\biggl(\frac{\RR_3(\tau_j)}{\PPP{j}{+}(\tau_j)}\biggr)^2
      +  \biggl(\frac{\RR_3(-\tau_j)}{\PPP{j}{-}(-\tau_j)}\biggr)^2
  =  \frac{[c_{3,0} \, \tau_j^3+ c_{1,2} \, \tau_j]^2}{2\,\tau_j^2(\tau_1^2 - \tau_2^2 )^2}
       +  \frac{[c_{2,1} \, \tau_j^2 + c_{0,3} ]^2}{2\,\tau_j^2(\tau_1^2 - \tau_2^2 )^2}
  =  \I_j + \II_j \, ,
\]
for $j=1,2$.

Since
\begin{eqnarray*}
\tau_1^2 + \tau_2^2 = 2\,b
&  \text{and}  &
\tau_1^2 \, \tau_2^2 = a \, ,  \\
\noalign{\hspace*{-\parindent}
we have}
(\tau_1^2 - \tau_2^2)^2 = 4\,\delta
&  \text{and}  &
\tau_1^4 + \tau_2^4 = 2\,(b^2 + \delta) \, ,
\end{eqnarray*}
hence
\begin{align*}
\I_1 + \I_2
&  =  \frac{[c_{3,0} \, \tau_1^2+ c_{1,2}]^2 + [c_{3,0} \, \tau_2^2+ c_{1,2}]^2}{8\,\delta}  \\
&  =  \frac{c_{3,0}^2 \, [ \tau_1^4 + \tau_2^4 ] + 2\,c_{1,2} \, c_{3,0} [ \tau_1^2 + \tau_2^2 ] + 2 \, c_{1,2}^2}{8\,\delta}  \\
&  =  \frac{2 \, c_{3,0}^2 \, (b^2 + \delta) + 4 \, c_{1,2} \, c_{3,0} \, b + 2 \, c_{1,2}^2}{8\,\delta}  \\
&  =  \frac{[b \, c_{3,0} + c_{1,2}]^2}{4\,\delta} + \frac{1}{4} \, c_{3,0}^2 \, ,
\end{align*}
and
\begin{align*}
\II_1 + \II_2
&  =  \frac{[c_{2,1}\tau_1^2 + c_{0,3}]^2}{8 \, \tau_1^2 \delta}
      +  \frac{[c_{2,1}\tau_2^2 + c_{0,3}]^2}{8 \, \tau_2^2 \delta}  \\
&  =  \frac{\tau_2^2 \, [c_{2,1}\tau_1^2 + c_{0,3}]^2 + \tau_2^1 \, [c_{2,1}\tau_2^2 + c_{0,3}]^2}
           {8 \, \tau_1^2 \tau_2^2 \delta}  \\
&  =  \frac{\tau_1^2\tau_2^2 \, (\tau_1^2+\tau_2^2) \, c_{2,1}^2
              + 4 \, c_{2,1} \, c_{0,3} \, \tau_1^2\tau_2^2
              + c_{0,3} \, (\tau_1^2+\tau_2^2)}
           {8 \, \tau_1^2 \tau_2^2 \delta}  \\
&  =  \frac{2\,a\,b \, c_{2,1}^2
              + 4 \, a \, c_{2,1} \, c_{0,3}
              + 2\,b \, c_{0,3}}
           {8 \, a \, \delta}  \\
&  =  \frac{a\,b^2 \, c_{2,1}^2
              + 2 \, a \, b \, c_{2,1} \, c_{0,3}
              + b^2 \, c_{0,3}}
           {4 \, a \, b \, \delta}  \\
&  =  \frac{a\,[b \, c_{2,1} + c_{0,3} ]^2
              + (b^2-a) \, c_{0,3}}
           {4 \, a \, b \, \delta}  \\
&  =  \frac{[b \, c_{2,1} + c_{0,3} ]^2}
           {4 \, b \, \delta}
   +  \frac{1}{4} \, \frac{c_{0,3}}{a\,b} \, .
\end{align*}

In conclusion
\begin{align*}
\ell_1(\xi)^2 +  & \ell_2(\xi)^2 + \ell_3^2(\xi) + \ell_4^2(\xi)  \\
&  =  \biggl(\frac{\RR_3(\tau_1)}{\PPP{1}{+}(\tau_1)}\biggr)^2
      +  \biggl(\frac{\RR_3(-\tau_1)}{\PPP{1}{-}(-\tau_1)}\biggr)^2
      +  \biggl(\frac{\RR_3(\tau_2)}{\PPP{2}{+}(\tau_2)}\biggr)^2
      +  \biggl(\frac{\RR_3(-\tau_2)}{\PPP{2}{-}(-\tau_2)}\biggr)^2  \\
&  =  \frac{1}{4} \,
      \biggl[\frac{[b(\xi) \, c_{2,1}(\xi)  + c_{0,3}(\xi) ]^2}{b(\xi) \, \delta(\xi)}
             + \frac{c_{0,3}^2(\xi)}{a(\xi) \, b(\xi)}
             + \frac{[c_{1,2}(\xi)+b(\xi) \, c_{3,0}(\xi) ]^2}{\delta(\xi)}
             + c_{3,0}^2(\xi)\biggr] \, ,
\end{align*}
hence the boundness of~the $\ell_j$s is equivalent
to~\eqref{E-Lcc1}--\eqref{E-Lcc3}.
\end{proof}

\subsubsection*{Case~II. $\xi \in \Omega_2$: $a(\xi) > 0$ and $\delta(\xi) = 0$}

In this case $\tau_1(\xi)=\tau_2(\xi)=\sqrt{b(\xi) \, }$ with $b(\xi)\ne0$;
\eqref{E-Lcc1}--\eqref{E-Lcc4} reduce to
\begin{align}
b(\xi) \, c_{3,0}(\xi) + c_{1,2}(\xi)
&  \equiv  0  \labelp{E-Lcc1}  \\
b(\xi) \, c_{2,1}(\xi) + c_{0,3}(\xi)
&  \equiv  0  \labelp{E-Lcc2}  \\
\bigl[c_{0,3}(\xi)\bigr]^2
&  \lesssim  b^3(\xi) \labelp{E-Lcc3}  \\
\bigl[c_{2,1}(\xi)\bigr]^2
&  \lesssim  b(\xi) \labelp{E-Lcc4} \, .
\end{align}
Note that condition~\eqref{E-Lcc3p} is a~consequence of~\eqref{E-Lcc2p} and~\eqref{E-Lcc4p}.

{}From the identity
\begin{align*}
c_{3,0}(\xi) \, \tau^3 +  &  c_{2,1}(\xi) \, \tau^2 + c_{1,2}(\xi) \, \tau + c_{0,3}(\xi)  \\
&  =  \frac{1}{2} \,
      \Bigl(c_{3,0}(\xi) + \frac{c_{2,1}(\xi) }{\sqrt{b(\xi) \, }}\Bigr) \,
      \Bigl[ \tau+\sqrt{b(\xi) \, } \Bigr]
      \bigl[ \tau^2-b(\xi) \bigr]  \\
&  \qquad
      +  \frac{1}{2} \,
         \Bigl(c_{3,0}(\xi) - \frac{c_{2,1}(\xi) }{\sqrt{b(\xi) \, }}\Bigr) \,
         \Bigl[ \tau-\sqrt{b(\xi) \, } \Bigr]
         \bigl[ \tau^2-b(\xi) \bigr]  \\
&  \qquad  +  (b(\xi) \, c_{3,0}(\xi) + c_{1,2}(\xi)) \, \tau + b(\xi) \, c_{2,1}(\xi) +c_{0,3}(\xi) \, ,
\end{align*}
we see that if~\eqref{E-Lcc1p},~\eqref{E-Lcc2p} and~\eqref{E-Lcc4p} are~satisfied,
then~we have the decomposition~\eqref{E-propdec3} with bounded coefficients.

On the other side, as $\tau_1^2(\xi)=\tau_2^2(\xi)=b(\xi)$,
the reduced polynomials of~$\PP(\tau,\xi)$ are all divisible by~$\tau^2-b(\xi)$,
hence in~order to have the decomposition~\eqref{E-propdec3}
$\RR_3(\tau,\xi)$ should be divisible by $\tau^2-b(\xi)$ too.
{}From the identity
\begin{align*}
c_{3,0}(\xi) \, \tau^3 +  &  c_{2,1}(\xi) \, \tau^2 + c_{1,2}(\xi) \, \tau + c_{0,3}(\xi)  \\
&  =  \bigl(c_{3,0}(\xi) \, \tau + c_{2,1}(\xi)\bigr) \, \bigl[ \tau^2-b(\xi) \bigr]  \\
&  \qquad  +  \bigl(b(\xi) \, c_{3,0}(\xi)+c_{1,2}(\xi)\bigr) \, \tau
      +  b(\xi) \, c_{2,1}(\xi)+c_{0,3}(\xi)
\end{align*}
we see that this means that~\eqref{E-Lcc1p} and~\eqref{E-Lcc2p} should be~satisfied.

If~\eqref{E-Lcc1p} and~\eqref{E-Lcc2p} hold true,
identity~\eqref{E-propdec3} reduces to
\[
\bigl(c_{3,0}(\xi) \, \tau + c_{2,1}(\xi)\bigr) \, \bigl[ \tau^2-b(\xi) \bigr]  \\
  =  \ell_1(\xi) \, \Bigl[ \tau-\sqrt{b(\xi) \, } \Bigr]\bigl[ \tau^2-b(\xi) \bigr]
     + \ell_2(\xi) \, \Bigl[ \tau+\sqrt{b(\xi) \, } \Bigr]\bigl[ \tau^2-b(\xi) \bigr] \, .
\]
Dividing by $\tau^2-b(\xi)$ and letting $\tau\to\sqrt{b(\xi) \, }$
we~see that,
in order to have the decomposition~\eqref{E-propdec3},
\eqref{E-Lcc4p} should be~satisfied too.

\subsubsection*{Case~III. $\xi \in \Omega_3$: $a(\xi) = 0$ and $\delta(\xi) > 0$}

In this case  $\tau_1(\xi)=\sqrt{2 \, b(\xi) \, } \ne 0$, $\tau_2(\xi) = 0$,
and~\eqref{E-Lcc1}--\eqref{E-Lcc3} reduce to
\begin{align}
\bigl|c_{1,2}(\xi)\bigr|
&  \lesssim  b(\xi) \labelpp{E-Lcc1}  \\
\bigl|c_{2,1}(\xi)\bigr|
&  \lesssim  \sqrt{b(\xi) \, } \labelpp{E-Lcc2}  \\
c_{0,3}(\xi)
&  = 0  \labelpp{E-Lcc3} \, .
\end{align}
Condition~\eqref{E-Lcc4} is a~consequence of~\eqref{E-Lcc2pp} and~\eqref{E-Lcc3pp}.

{}From the identity
\begin{align*}
c_{3,0}(\xi) \, \tau^3 +  &  c_{2,1}(\xi) \, \tau^2 + c_{1,2}(\xi) \, \tau + c_{0,3}(\xi)  \\
&  =  \frac{1}{2} \,
      \Bigl(c_{3,0}(\xi) - \frac{c_{2,1}(\xi) }{\sqrt{2 \, b(\xi) \, }} + \frac{c_{1,2}(\xi)}{2 \, b(\xi)}\Bigr) \,
      \tau^2 \, \Bigl[ \, \tau-\sqrt{2 \, b(\xi) \, } \, \Bigr]  \\
&  \qquad  +  \frac{1}{2} \,
              \Bigl(c_{3,0}(\xi) + \frac{c_{2,1}(\xi) }{\sqrt{2 \, b(\xi) \, }}
                                 + \frac{c_{1,2}(\xi)}{2 \, b(\xi)}\Bigr) \,
              \tau^2 \, \Bigl[ \, \tau+\sqrt{2 \, b(\xi) \, } \, \Bigr]  \\
&  \qquad  -  \frac{c_{1,2}(\xi)}{2 \, b(\xi)} \, \tau \, \bigl(\tau^2-2 \, b(x)\bigr) + c_{0,3}(\xi) \, .
\end{align*}
we see that if~\eqref{E-Lcc1pp}--\eqref{E-Lcc3pp} are satisfied,
then~\eqref{E-propdec3} holds true and the coefficients $\ell_j$ are bounded.

On the other side,
as all the reduced polynomials of~$\PP(\tau,\xi)$ are divisible by~$\tau$,
in order to have the decomposition~\eqref{E-propdec3},
$\RR_3(\tau,\xi)$ should be divisible by $\tau$ too; this means $c_{0,3}(\xi)=0$.
Dividing~\eqref{E-propdec3} by~$\tau$,
we~get
\begin{align*}
c_{3,0}(\xi) \, \tau^2 +  &  c_{2,1}(\xi) \, \tau + c_{1,2}(\xi)  \\
&  =  \ell_1(\xi) \, \tau \, \Bigl[ \, \tau-\sqrt{2 \, b(\xi) \, } \, \Bigr]
     + \ell_2(\xi) \, \tau \, \Bigl[ \, \tau+\sqrt{2 \, b(\xi) \, } \, \Bigr]
     + \ell_3(\xi) \, \bigl(\tau^2-2 \, b(x)\bigr) \, ;
\end{align*}
that means that the polynomial on the l.h.s.\
has a proper decomposition w.r.t.~the polynomial $\tau^3-2 \, b(\xi) \, \tau$.
Repeating the same argument of~the proof of~Lemma~\ref{L-R2},
we~can show the necessity of~\eqref{E-Lcc1pp} and~\eqref{E-Lcc2pp}.

\subsubsection*{Case~IV. $\xi \in \Omega_4$: $a(\xi) = \delta(\xi) = 0$}

In this case  $\tau_1(\xi)=\tau_2(\xi)=0$,
and~\eqref{E-Lcc1}--\eqref{E-Lcc3} reduce to
\begin{equation} \label{E-pd3IV}
c_{2,1}(\xi) = c_{1,2}(\xi) = c_{0,3}(\xi) = 0 \, .
\end{equation}

As all the reduced polynomials of~$\PP(\tau,\xi)$ are equal to~$\tau^3$,
we have the decomposition~\eqref{E-propdec3}
if, and only if, $\RR_3(\tau,\xi)$ is divisible by $\tau^3$,
which means that~\eqref{E-pd3IV} is satisfied.

\section*{Acknowledgements}

Both authors are members of the
Gruppo Nazionale per l'Analisi Matematica, la Probabilit\`{a} e~le loro Applicazioni
(GNAMPA-INDAM).

The first author~has been partially supported by the PRIN 2022 project
``Anomalies in partial differential equations and applications''
CUP~D53C24003370006 .

\providecommand{\bysame}{\leavevmode\hbox to3em{\hrulefill}\thinspace}
\renewcommand{\MR}[1]{}

\end{document}